\documentclass[10pt,fleqn]{article}
\usepackage{amsmath,amssymb,amsfonts,amsthm,color,dsfont,graphicx}
\usepackage[a4paper,margin=3cm]{geometry}
\usepackage[bf,SL,BF,hang]{subfigure}
\usepackage[english]{babel}
\usepackage[utf8]{inputenc}

\graphicspath{{./figures/}}
\allowdisplaybreaks

\usepackage[pdftex]{hyperref}
\hypersetup{
    colorlinks = true,
    citecolor=black,
    filecolor=black,
    linkcolor=black,
    urlcolor=black,
    pdftitle={A posteriori modeling error estimates in the optimization of two-scale elastic composite materials},
    bookmarksopen=true,
    pdfdisplaydoctitle
}

\newcommand{\R}{\mathds{R}}
\newcommand{\Z}{\mathds{Z}}

\newcommand{\tdiff}[1]{\frac{\mathsf{d}}{\mathsf{d} #1}}
\renewcommand{\d}{\;\mathsf{d}}
\newcommand{\D}{\ensuremath{\mathrm{D}}}

\newcommand{\body}{\Omega}
\newcommand{\workdom}{D}
\newcommand{\intga}{\int_{\Gamma_N}\,}
\newcommand{\intbo}{\int_{\workdom}\,}
\newcommand{\normal}{\nu}

\newcommand{\strain}{u}
\newcommand{\lstrain}{u^\mathrm{L}}
\newcommand{\sstrain}{u^\mathrm{S}}
\newcommand{\sstrainh}{u^\mathrm{S}_h}
\newcommand{\lstrainh}{u^\mathrm{L}_h}
\newcommand{\lstrainhk}{u^\mathrm{L}_{h,k}}
\newcommand{\tlstrainhk}{\tilde u^\mathrm{L}_{h,k}}
\newcommand{\stress}{\sigma}
\newcommand{\eps}[1]{\varepsilon[#1]}
\newcommand{\etensor}{C}
\newcommand{\efftensor}{\etensor^\ast}
\newcommand{\ltensor}{\etensor^\mathrm{L}}
\newcommand{\ltensorh}{C^\mathrm{L}_h}
\newcommand{\ltensorhk}{C^\mathrm{L}_{h,k}}
\newcommand{\tltensorhk}{\tilde C^\mathrm{L}_{h,k}}
\newcommand{\stensor}{\etensor^\mathrm{S}}
\newcommand{\stensorh}{\etensor^\mathrm{S}_h}

\newcommand{\Vh}{\mathcal{V}_h}

\newcommand{\Lag}{\mathcal{L}}

\newcommand{\ie}{i.\,e.}
\newcommand{\eg}{e.\,g.}
\newcommand{\wrt}{w.\,r.\,t.}
\newcommand{\etal}{et\,al.\ }

\newcommand{\notinclude}[1]{}

\newcommand{\fr}[2]{{\textstyle\frac{#1}{#2}}}

\renewcommand{\div}{\mathrm{div}}

\newtheorem{theorem}{Theorem}[section]

\theoremstyle{definition}

\begin{document}

\title{A posteriori modeling error estimates in the optimization of two-scale elastic composite materials}

\author{
       Sergio Conti\footnotemark[1]
  \and Benedict Geihe\footnotemark[2]
  \and Martin Lenz\footnotemark[2]
  \and Martin Rumpf\footnotemark[2]
}

\renewcommand{\thefootnote}{\fnsymbol{footnote}}
\footnotetext[1]{Institute for Applied Mathematics, Rheinische Friedrich-Wilhelms-Universit\"at Bonn, Endenicher Allee 60, 53115 Bonn, Germany}
\footnotetext[2]{Institute for Numerical Simulation, Rheinische Friedrich-Wilhelms-Universit\"at Bonn, Endenicher Allee 60, 53115 Bonn, Germany}
\renewcommand{\thefootnote}{\arabic{footnote}}

\maketitle

\begin{abstract}
The a posteriori analysis of the discretization error and the modeling error is studied
for a compliance cost functional in the context of the optimization of composite elastic materials
and a two-scale linearized elasticity model.
A mechanically simple,  parametrized microscopic supporting structure is chosen
and the parameters describing the structure are determined minimizing the compliance objective.
An a posteriori error estimate is derived which includes the
modeling error caused by the replacement of a nested laminate microstructure by this considerably simpler microstructure.
Indeed, nested laminates  are known to realize the minimal compliance  and provide a benchmark for the quality of the microstructures.
To estimate the local difference in the compliance functional the dual weighted residual approach is used.
Different numerical experiments show that the resulting adaptive scheme leads to simple parametrized microscopic supporting structures
that can compete with the optimal nested laminate construction.
The derived a posteriori error indicators allow to verify that
the suggested simplified microstructures achieve the optimal value of the compliance up to a few percent.
Furthermore, it is shown how discretization error and modeling error can be balanced
by choosing an optimal level of grid refinement.
Our two scale results with a single scale microstructure can provide guidance towards
the design of a producible macroscopic fine scale pattern.
\end{abstract}

\section{Introduction}\label{sec:intro}
It is a fundamental insight in shape optimization that elastic bodies subjected to external surface loadings exhibit
the spontaneous development of microstructures, for example in the case of compliance optimization \cite{Al02}. This phenomenon is a
practical manifestation of the fact that the associated minimization problem is in general ill-posed.
One popular solution appeals to relaxation theory, thereby allowing composite materials
with intermediate density and effective elasticity tensors resulting from a microscopic mixture of the involved constituents.  Homogenization theory \cite{DoCi99,Mi02}
is then used to derive the effective material properties of a composite elastic material, starting from
a given microscopic decomposition into homogeneous regions of the individual ingredient materials.
Using the G-closure theory \cite{Cherkaev2000,Mi02}  it can be shown that a  nested lamination structure realizes the minimal value of the compliance functional \cite{Al02}.
Let us note that optimal constructions exist but are not unique. In fact, several entirely different
microstructures lead to globally optimal material properties at the macroscopic level \cite{Cherkaev2000,Mi02}.
Most of them are however purely theoretical, lack a truly mechanical
equivalent and would not be manufacturable.
We address this optimal design problem for 2D elastic material composites
numerically and introduce a microscopic supporting structure consisting of two orthogonal
trusses with varying thickness and arbitrary rotation (cf. Figure~\ref{fig:microsketch}).
For these simple microscopic cell geometries a collocation type boundary element method is employed to compute the corrector problem in a two-scale ansatz and to
evaluate the local effective elasticity tensor. This is then combined with a bi-quadratic
finite element scheme on a rectangular grid on the macroscale to set up a two-scale approach
for the simulation of microstructured composite elastic materials.
We shall in particular focus on a microscopic construction with rotated, orthogonal
trusses with varying thickness, and show that it leads to material composites with a compliance cost
remarkably close to the optimal composite obtained from sequential lamination.
Numerical experiments with different simple, parametrized geometries in the microscopic cells
were performed in \cite{CoGeRu14}, and have lead to the identification of the rotated truss construction as the most promising candidate.
This microscopic pattern has already been extensively studied in the literature (see the discussion below).
Based on this two-scale approach for the elastic state equation one can optimize the parameters controlling the microscopic pattern,
\ie, the two width parameters of the trusses and the
rotation angle, considered as functions of the macroscale position. Numerical results suggest that
the parameter functions controlling the microscopic patterns are smooth
in large regions of the elastic workpiece. At the same time,   in small regions
the simple model for the microscopic supporting structure tends
to show more complex patterns.

\emph{Related work.}
Shape optimization in general is a well-established field and covered extensively in the literature, see \eg\ the textbooks
\cite{Be95a,Al02}. Regarding the theory of homogenization we refer to \cite{Jikov1994,BrDe98,DoCi99,BuDa91,Cherkaev2000,Mi02}. The foundations
for the sequential laminates construction were already laid early in \cite{Ta85,MuTa85,FrMu86,LuCh86,GiCh87,Av87}.
Jog \etal\ \cite{JogHaberBendsoe1994}
and Allaire \etal\ \cite{AlBoFr97} used the nested lamination ansatz  to develop a  practical numerical scheme for the optimization of the material composition in mechanical work pieces.
Thus, these methods approximate the truly optimal material microstructure.

Elastic shape optimization with geometrically simple microscopic structures
has first been investigated by Bends{\o}e and Kikuchi \cite{BeKi88}.
They in particular already considered rotated cells with a rectangular hole (see \cite[Section 3]{BeKi88}) and applied adaptive meshes to
resolve the macroscopic variation of the microstructure parameters.
Optimal microstructures based on these rotated truss geometries are also studied by Kikuchi and Suzuki \cite{SuKi91} for instance for the cantilever design problem.
The same quasi periodic microscopic pattern was applied by Rodrigues and Fernandes \cite{RoFe94} for thermoelastic material optimization also using an adaptive finite element method.
Here, we pick up this approach and provide a posteriori estimates for the discretization error and the modeling error.
Surely, due to the explicit relation between the lamination parameters and the effective elasticity tensor, algorithms based on the
nested lamination approach are very efficient and need a numerical discretization solely on the macroscale.
The intention of this paper is to study in particular the modeling error that arises when replacing the nested laminate with its $3$ different scales in 2D by the two-scale model with the
rotated trusses on the microscale. Surely, this comes at the expense of the solution of the microscopic PDE problem.
In fact, we demonstrate  exemplarily that the a posteriori error estimation of modeling errors of this type is accessible via the dual weighted residual approach.
To this end, we combine a finite element scheme on the macroscale as for instance in \cite{BeKi88,SuKi91} with a boundary element scheme on the microscale.
A priori error estimates for a finite element scheme in two-scale PDE-constrained optimization were very recently presented by Li \etal\ \cite{LiCa15},
were the PDE is a diffusion equation in a domain with periodic microstructure.

Further optimal microstructures include concentric spheres for hydrostatic loads \cite{Ha62}, confocal ellipsoids
\cite{GrKo95}, and the Vigdergauz construction \cite{Vi94}. We refer to
\cite{Cherkaev2000,Mi02} for a more complete discussion.
Related to our two-scale approach with parametrized microstructures \cite{CoGeRu14} is also the earlier work by Barbarosie and Toader.
Based on optimization of holes in generalized periodic domains for given macroscopic affine displacements \cite{BaTo10,BaTo10a}
they set up a macroscopic finite element scheme in \cite{BaTo12} working with effective material properties evaluated via
numerical homogenization.
All these methods are in line with the heterogeneous multi-scale method (HMM) \cite{EEnHu03,EEn03,EEn05,EMiZh05}, a general paradigm
to tackle numerical problems involving different length scales.

A different approach to the optimization of microscopic geometries is the  free material optimization approach,
where one first aims at the locally optimal elasticity tensor. Bends{\o}e \etal\ showed in \cite{BeGu94} that one is lead to orthotropic materials with
directions of orthotropy aligned with the directions of principal strains. Furthermore, they computed numerical optimal material properties.
Haslinger \etal\ used this ansatz to numerically compute a distribution of optimal elasticity tensors in 3D \cite{HaKoLe10}.
In this case one first optimizes the coefficients of the effective elasticity tensor, and then in a post-processing step
one searches for approximating microscopic realizations.
Our approach is also related to the post processing method by Pantz and Trabelsi \cite{PaTr08}
who used a local truss construction which
is  deformed based on the optimal lamination directions (cf. \cite[Section 4.1]{PaTr08}).
Therefore, as in our approach the nested microstructure construction is avoided. Indeed, their structures do not have
any inconsistency at macroscopic cell boundary and
can in principle be manufactured.

A posteriori error estimates for elliptic homogenization problems with fine scale diffusion were derived
in the context of HMM in \cite{Oh05,HeOh09}.
In contrast to classical discretization error estimates it is however often required to assess the error \wrt\ a certain cost
functional. Such goal-oriented error estimates for quantities of interest of a real composite work piece were derived in
\cite{PrOd99,OdVe00,OdVe00a,Ve04}.
In this contribution we follow the dual weighted residual (DWR) method \cite{BeRa97}, see \eg\ \cite{BeKaRa00} in particular for optimal control problems.
This approach has recently been put to work for a variety of different applications.
See \cite{BeEsTr11} for a quasi-optimality result of the adaptive finite element scheme using a special marking strategy,
\cite{BeVe09,VeWo08,LeMeVe13} for the treatment of control and state constraints,
and \cite{KoLe13} for matrix valued $L^1$ optimal control.
A posteriori error estimates in the context of shape optimization have been studied in  \cite{KiVe13} for
a tracking type cost functional and Helmholtz state equation,
in \cite{MoNoPa10,MoNoPa12}, where DWR is used to assess the PDE error while the geometric error estimates relate to the
Laplace Beltrami operator,
and in \cite{KaDeGa14, Ka13} for one shot methods applied to fuel ignition problems and aerodynamic shape optimization.
In \cite{BrLiUl12} error estimation for the optimal design in the context of Navier Stokes flow is discussed and in \cite{Wo10} estimates for the variable thickness sheet model are derived.
In \cite{GeRu15} a posteriori error estimates have been derived for shape optimization in a two-scale context with nested laminates on the microscale.
To this end, the dual weighted residual approach has been applied leading to an associated adaptive meshing strategy.

\medskip
The paper is organized as follows. In Section~\ref{sec:basics} we will briefly address the fundamental concepts of linearized elasticity, shape optimization \wrt\
the compliance objective and optimal material composites attained by the sequential lamination construction. Furthermore, we briefly review the numerical two-scale model in linearized elasticity.
The dual weighted residual approach for the compliance cost is developed in Section~\ref{sec:dwr}. In the resulting estimates, as usual, certain weights still involve the continuous solution of the state equation.
A suitable numerical approximation which enables
to derive effective local error indicators is presented in Section~\ref{sec:weighting}. Some comments on the implementation are given in Section~\ref{sec:impl}
and in Section~\ref{sec:results} we present our numerical results.

\section{Optimization of material composites}\label{sec:basics}
In this section we will briefly revise the fundamental concepts of linearized elasticity, shape optimization with the compliance objective and optimal microstructures
given by sequential lamination. Furthermore, we will introduce a two-scale approach for the elastic behavior of microstructured materials,
where the  microstructure is given by a parametrized truss construction.

\emph{Linearized elasticity.}
Let us assume that an elastic workpiece is described by a simply connected domain $\workdom \subset \R^2$ with Lipschitz boundary $\partial \workdom$.
Suppose that the boundary is split into a fixed relatively open subset $\Gamma_D$, a Dirichlet part where the displacement vanishes,
and $\Gamma_N := \partial D \setminus \Gamma_D$, a Neumann part where sufficiently regular surface loads $g$ can be applied.
Then the induced displacement $\strain[C]$  is  the unique solution $u: \workdom \rightarrow \R^2$ of the partial differential equations
of linearized elasticity, given in variational form as
\begin{equation}\label{eq:weak}
 a(\etensor;\strain,\varphi) = l(\varphi) \;\;\forall\,\varphi \in H^1_{\Gamma_D}
\end{equation}
with the quadratic form $a(\etensor;\strain,\varphi) := \intbo \etensor(x) \, \eps{\strain} : \eps{\varphi} \d x$ and the linear form
$l(\varphi) := \intga g \cdot \varphi \d a(x) \,$.
Here $H^1_{\Gamma_D}$ denotes the Sobolev space of $L^2$ vector-valued functions with weak derivatives in $L^2(\workdom)$
and vanishing trace on $\Gamma_D$, $\eps{\strain} := \frac12 \left( \D\strain + \D\strain^\top \right)$ denotes the symmetrized strain tensor
with the Jacobian $\D \strain$, $\normal$ the outward pointing normal on $\Gamma_N$ and $\etensor \in L^\infty(\workdom,\R^{2^4})$ the elasticity
tensor. Furthermore, $(\etensor \epsilon)_{ij}  = \sum_{kl} \etensor_{ijkl} \epsilon_{kl}$ and $\sigma:\epsilon = \sum_{ij} \sigma_{ij}\epsilon_{ij}$ for two matrices $\sigma,\, \epsilon \in \R^{2\times2}$.
The fourth order tensor $\etensor(x)$ characterizes the material properties at each point $x \in \workdom$ and will later be given by the effective
tensor arising either from the explicit homogenization formula for laminate microstructures or from the solution of
the microscopic cell problem. As usual, we assume that $\etensor$ fulfills the symmetry properties
$\etensor_{ijkl} = \etensor_{jikl} = \etensor_{ijlk} = \etensor_{klij}$
and the ellipticity condition
$\etensor \, \xi : \xi \geq c \, |\xi+ \xi^\top|^2$.

\emph{Shape optimization.}
We consider a cost functional $J[\etensor,\strain]$ which we aim at minimizing under the constraint that $\strain$ solves the associated elastic problem \eqref{eq:weak}.
In our numerical application, we adopt the classical compliance optimization approach to shape optimization. This means that we
 optimize the rigidity of an object by minimizing the elastic energy, or compliance, given by the functional
\begin{equation} \label{eq:compliance}
 J[\etensor,\strain] := \intga g \cdot \strain\d a(x) = l(\strain ) \,,
\end{equation}
which does not depend explicitly on $\etensor$. However,
from (\ref{eq:weak}) one obtains $J[\etensor,\strain[\etensor] ]=a(\etensor;\strain[\etensor] ,\strain[\etensor] ) $.
The shape optimization problem now amounts to finding a subset $\body \subset \workdom$ where a hard material with
elasticity tensor $A$ should be placed.
The remaining part $\workdom \setminus \body$ is either left void or filled with a weak material with
elasticity tensor $B$,
so that the actual elasticity tensor is given by $\etensor = \chi_\body A + (1-\chi_\body) B$, where
$\chi \in L^\infty(\workdom,\lbrace 0,1 \rbrace)$
 is the characteristic function of $\body$. Throughout this paper we only consider isotropic materials that are described by the two
Lam\'e parameters $\lambda$ and $\mu$.
If the amount of hard material is constrained, \ie\ $ \int_\workdom \chi_{\Omega} \d x = \Theta$ with $\Theta > 0$ fixed, the resulting minimization problem
is ill-posed and the formation of microstructures can be observed in numerically computed minimizing sequences.

\emph{Sequential lamination.}
To recover well-posedness the minimization problem can be relaxed by allowing material of intermediate density at each point, \ie\ $\chi \in [0,1]$.
The theory of homogenization then permits to compute the effective material properties $\efftensor$ of mixtures of the constituents $A$ and $B$ on
different underlying length scales. We will  refer to the literature for details \cite{Braides1998,Al02,DoCi99,Mi02}.\\
Most important for our exposition here is the fact that a sequential lamination construction yields effective materials that attain lower bounds
on the local elastic energy density, called Hashin-Shtrikman bounds, and thus represent optimal microstructures. They are obtained by layering hard
and soft material along a certain direction with a certain ratio, computing homogenized elastic properties and using those to iterate the
construction on a subsequent, substantially larger length scale, see the sketch in Figure~\ref{fig:microsketch}. For the compliance objective in two
space dimensions two iterations of the lamination construction are indeed sufficient and the parameters of the construction (direction of the lamination, volume fraction of each phase, and overall local density)
can explicitly be computed from the local stress
$\sigma(x) = \efftensor(x) \eps{\strain(x)}$. At the same time the effective material properties can likewise be computed explicitly from the
parameters, leading to an alternating algorithm for computing globally optimal relaxed shapes.
Such a numerical method for the two-scale optimization with nested laminates and an alternating descent algorithm was presented by Jog \etal\ \cite{JogHaberBendsoe1994}.
It is moreover possible to ultimately pass from the
weak material $B$ to void. For further details we refer to \cite{KoLi88,AlKo93a,AlKo93,AlBoFr97}.

\emph{A two-scale approach for approximating microstructures.}
In this paper, we aim at the numerical computation of a near optimal material composite which consists of a simple microstructure, which is at least locally
mechanically constructible. For this microstructure we choose a model which has already been proposed in the $80$s  \cite{BeKi88,SuKi91}
based on a microscopic pattern of two rotated, orthogonal trusses of different width, see Figure~\ref{fig:microsketch} and the discussion above.
To this end, assume that the microstructure of the two orthogonal, rotated trusses is parametrized via the (relative) width $\delta_1$ and $\delta_2$ of the two trusses ($0<\delta_1,\,\delta_2 <1$)
and the rotation angle $\alpha$. The vector of parameters is denoted by $q = (\alpha, \delta_1, \delta_2)$ and depends on the macroscopic position $x$, $q(x):= (\alpha(x), \delta_1(x), \delta_2(x))$.
We then define the microscopic pattern on the (periodically extended) fundamental cell $\omega[q(x)] := Q(\alpha(x))(-\tfrac12,\tfrac12)^2$ of a periodic lattice at the position $x$, with $Q(\alpha)$ denoting the rotation by the angle $\alpha$.
In fact, the cell $\omega[q(x)]$  splits  into
a domain $\omega_A[q(x)]$ with hard material described by the elasticity tensor $A$ and a remaining domain
$\omega_B[q(x)] := \omega[q(x)] \setminus \omega_A[q(x)]$ with soft material described by the elasticity tensor $B$.
The hard phase is given by
\begin{equation*}
 \omega_A[q(x)] := Q(\alpha(x)) \left( [-\tfrac{\delta_1(x)}{2},\tfrac{\delta_1(x)}{2}] \times [-\tfrac{1}{2},\tfrac{1}{2}] \cup [-\tfrac{1}{2},\tfrac{1}{2}] \times [-\tfrac{\delta_2(x)}{2},\tfrac{\delta_2(x)}{2}]\right)
\,.
\end{equation*}
Let us remark that in the implementation of the boundary element method used to solve the microscopic correction problem it is advantageous to consider a shifted fundamental cell $\tilde \omega[q(x)]$
resulting from a shift on the periodic lattice by $\tfrac12 Q(\alpha(x)) (1,1)$, with the
splitting $\tilde \omega_B[q(x)] = Q(\alpha(x)) \left( [\tfrac{\delta_1(x)}{2},1-\tfrac{\delta_1(x)}{2}] \times  [\tfrac{\delta_2(x)}{2},1-\tfrac{\delta_2(x)}{2}]\right)$ and
$\tilde \omega_A[q(x)] = \tilde \omega[q(x)] \setminus \tilde \omega_B[q(x)]$ which describe the identical periodic, microscopic pattern (cf. dotted red line in Figure~\ref{fig:microsketch}) with a single interface.
\begin{figure}[t]
\hfill
\includegraphics[height=.2\linewidth]{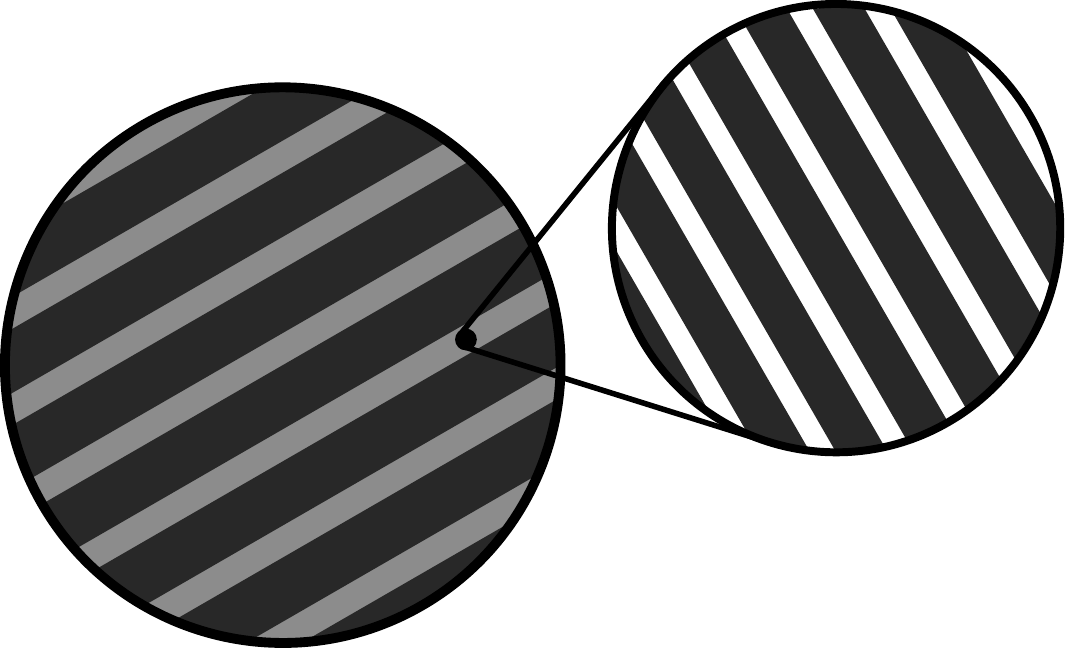}
\hfill
\includegraphics[height=.2\linewidth]{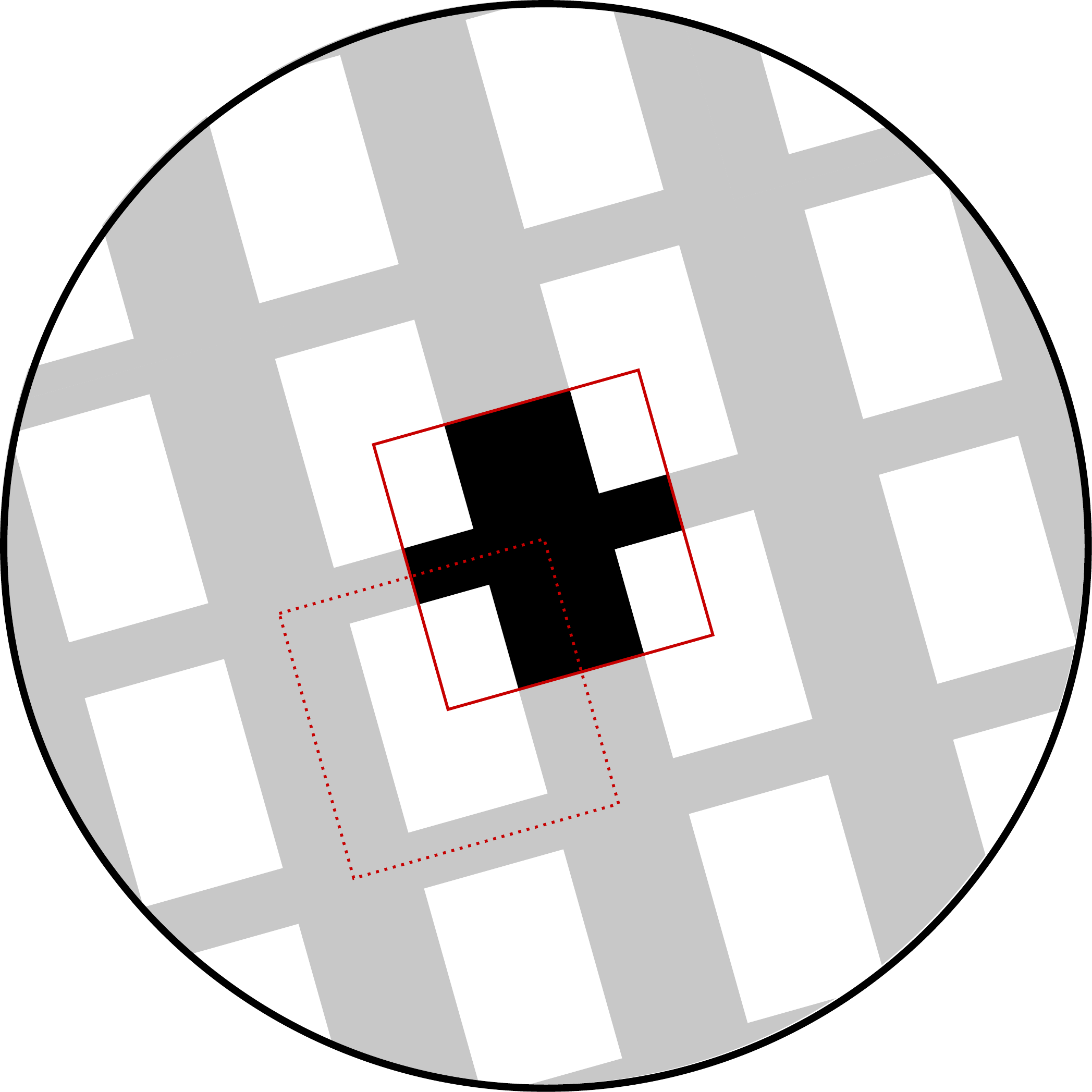}
\hfill\mbox{}
\caption{Sketch of twofold sequential lamination.
Freely rotated cells with rectangular holes (see dotted red marking) representing orthogonal trusses (see red marking).}
\label{fig:microsketch}
\end{figure}
The microscopic problem can be separated from the macroscopic one using the theory of homogenization
\cite{Braides1998,DoCi99,Al02}. The key modeling assumption here is that the truss microstructure is locally periodic
and much finer than the macroscopic degrees of freedom.
Given the
isotropic constituents $A$ and $B$ and admissible locally periodic perforations on the microscale,
which at a material point $x$ are represented by the parameters $q(x)$, one obtains the local effective elasticity
tensor $\efftensor(x)$ applied to any strain  $\xi\in\R^{2\times 2}_{sym}$ by the cell problem
(see \cite[Section 10.2]{DoCi99} or \cite[Theorem 14.7]{Braides1998})
\begin{equation}\label{eq:effectivevariational}
\efftensor[q](x) \, \xi:\xi = \int_{\omega[\alpha(x)]} C(x,y) \, \eps{\xi y + w} : \eps{\xi y + w} \d y
\end{equation}
with the $\omega[\alpha(x)]$-periodic corrector $w$ defined as the (up to constant translations unique) solution of
\begin{equation*}
\int_{\omega[\alpha(x)]}  \etensor(x,y) \left( \xi + \eps{w(x,y)} \right) : \eps{\psi(x,y)} \d y  = 0 \quad \forall \, \psi \in H^1_{per}(\omega[\alpha(x)];\R^2)
\end{equation*}
in the space $H^1_{per}(\omega[\alpha(x)];\R^2)$ of vector valued functions in $H^1$ on $\omega[\alpha(x)]$ with periodic boundary conditions.
Here, $\etensor(x,y) = A$ for $y \in \omega_A[q(x)]$  and $\etensor(x,y) = B$ for $y \in \omega_B[q(x)]$, and by uniqueness $\eps{w}$ depends linearly on $\xi$.
The tensor $\efftensor[q](x)$ describes the effective elastic properties of a material with an infinitesimally fine microstructure described by the three parameters $q(x)=(\alpha(x), \delta_1(x), \delta_2(x))$.
The macroscopic (homogenized) problem than takes the same form as in \eqref{eq:weak} with the local elasticity tensor  $C(x)$ replaced by the effective elasticity tensor $\efftensor[q](x)$,
\begin{equation*}
 a(\efftensor[q];\strain,\varphi) = l(\varphi) \;\;\forall\,\varphi \in H^1_{\Gamma_D}(D;\R^2) .
\end{equation*}
Practically,  based on \eqref{eq:effectivevariational} and the symmetry assumption for the effective elasticity tensor we can obtain the components of $\efftensor$ by
\begin{equation}  \label{eq:Cstar}
\efftensor_{ijkl} = \efftensor \varepsilon_{ij} : \varepsilon_{kl}
=  \efftensor \varepsilon_{ij+kl} : \varepsilon_{ij+kl} - \efftensor \varepsilon_{ij-kl} : \varepsilon_{ij-kl}
\end{equation}
with $\varepsilon_{ij} = \frac12 (e_i \otimes e_j + e_j \otimes e_i)$ and  $\varepsilon_{ij \pm kl} = \frac12 (\varepsilon_{ij} \pm \varepsilon_{kl})$
where $e_1 = (1,0)$ and $e_2 = (0,1)$.

Equivalently, we can formulate a single two-scale problem, which contains both the microscopic and the macroscopic degrees of freedom:
For given isotropic constituents $A$ and $B$ and admissible locally periodic perforations on the microscale parametrized by $q$,
find an effective macroscopic displacement $u^\star \in H^{1,2}_{\Gamma_D}$ and a periodic microscopic correction $w^\star \in \mathbf{W}_\alpha$
solving the two-scale equation (see for example \cite[Section 9.3]{DoCi99} for a derivation in a scalar situation)
\begin{equation*}
  \int_{\workdom} \int_{\omega[\alpha(x)]} \etensor(x,y)  (\eps{u^\star(x)} + \eps{w^\star(x,y)}) : (\eps{\phi(x)} + \eps{\psi(x,y)}) \d y \d x
= \int_{\Gamma_N} g(x) \cdot \phi(x) \d a(x)
\end{equation*}
for all $\phi \in H^{1,2}_{\Gamma_D}$ and all functions $\psi \in \mathbf{W}_\alpha$,
where the function space of microscopic periodic displacement corrections is defined as
\begin{align*}
\mathbf{W}_\alpha := \Big\lbrace &\psi: (x,y) \mapsto \psi(x,y) \in \R^2  \text{ measurable, where }\\
                                 &x\in \workdom,  \psi(x,y+z) = \psi(x,y) \, \forall z \in Q(\alpha(x)) \Z^2,
                                  \Vert \psi \Vert_{\mathbf{W}_\alpha} \leq \infty \Big\rbrace
\end{align*}
with $\|\psi\|_{\mathbf{W}_\alpha} := \left( \int_{\workdom}  \int_{\omega[\alpha(x)]} \psi(x,y)^2 + |\D_y \psi(x,y)|^2 \d y \d x  \right)^{\frac12}$.
As above, $\etensor(x,y) = A$ for $y \in \omega_A[q(x)]$  and $\etensor(x,y) = B$ for $y \in \omega_B[q(x)]$.
For the underlying two-scale methodology we refer to \cite{EMiZh05}.

\emph{Macroscopic and microscopic discretization.}
To discretize the problem we assume the macroscopic domain $D$ to be polygonally bounded and equipped with an admissible and regular finite element mesh $\mathcal{M}_h$, cf. \cite{Ciarlet1978},
with elements $E \in \mathcal{M}_h$ and a piecewise constant mesh size function $h$. In our implementation, we assume that $D$ can be meshed with  a rectangular mesh and
use a finite element ansatz for the discrete elastic displacement $\strain_h$ on the macroscale in the finite element space $\Vh$ of piecewise bi-quadratic and continuous vector-valued functions with vanishing trace on $\Gamma_D$.
As for the elastic material properties we consider a piecewise constant tensor field $\efftensor_h$, derived either from
a set of parameters describing sequential lamination or from microscopic perforation in the two-scale model.
We then compute the solution $\strain_h$ of the discrete weak problem $a(\efftensor_h;\strain_h,\varphi_h) = l(\varphi_h)$
for all $\varphi_h \in \Vh$.

In case the effective material properties $\efftensor_h$ are not given explicitly, as in our case of the rotated trusses, they have to be computed numerically by solving the cell problem \eqref{eq:effectivevariational}.
To this end, we employ a boundary element method for which only the boundaries of the perforated unit cell $\omega_A[q(x)]$ need to be discretized by polygon arcs.
Let us mention here that this microscopic mesh is discretized uniformly and will not undergo any refinement.
The boundary element method requires to deal with the fundamental solution of the linearized elasticity PDE. Here we choose as the underlying material model the Lam\'e Navier model for an isotropic material, see \eg\ \cite{St03},
\[
 u^\ast_{ki}(p,q) = \frac{\lambda + \mu}{4 \pi \left( \mu (\lambda +2 \mu ) \right) }
 \left( -\delta_{ki} \frac{\lambda+3\mu}{\lambda+\mu} \log\Vert p-q \Vert +
 \frac{(q_k-p_k)(q_i-p_i)}{\Vert p-q \Vert^2} \right),
\]
where $\lambda$ and $\mu$ are the Lam\'e parameters that will both be set to $1$ for our numerical computations.
The fundamental solution is used to rewrite the elasticity equation as a boundary integral equation leading to
\begin{equation}\label{eq:bem:BIE}
 w = U[A \eps{w} \cdot \normal] - V[w]\,,
\end{equation}
where the boundary integral operators $U$ and $V$ are the single and double layer operator, respectively.
On the discretized boundary a set of collocation points $\xi_i$ is fixed and the displacement $w$ as well as
the normal tension $A \eps{w} \cdot \normal$ is approximated by linear interpolation of nodal values at $\xi_i$.
Equation \eqref{eq:bem:BIE} now has to hold for every $\xi_i$ leading to a linear system of equations.
The boundary integral operators $U$ and $V$ are only applied to piecewise affine  functions on the boundary, therefore their application to the basis functions has been computed analytically.
As we typically consider mixed boundary value problems equation \eqref{eq:bem:BIE} has to be rearranged according to known
and unknown values. Further details can be found in \cite{AtCoGe12}.

The coefficients of the effective material tensor can then be evaluated as in \eqref{eq:effectivevariational} using a
boundary integral formulation. Likewise it is possible to derive a shape gradient, cf. \cite{DeZo01}, to be able to compute sensitivities
$\nabla_{q(x)} \efftensor_{ijkl}$ of the elastic coefficients \wrt\ the describing microscopic parameters.
It is then used in the optimization scheme for the minimization of the macroscopic cost functional \eqref{eq:compliance}.
Let use mention that deriving the cost functional initially also leads to sensitivities of the elastic solution \wrt\ $q(x)$.
We use as usual the associated Lagrangian approach and derive the adjoint equation to compute the gradient of the cost functional.
For further details we refer to the earlier work \cite{CoGeRu14}.
In Section~\ref{sec:impl} we will give further details on the implementation.

\section{Modeling and discretization estimate} \label{sec:dwr}
In what follows we will derive an estimate for the difference of the achievable compliance cost using the numerical two-scale model with a macroscopically
parametrized microscopic pattern of two rotated, orthogonal trusses and the optimal compliance cost  associated with the nested laminate construction.
This estimate reflects both
the modeling error caused by the choice of the mechanically simple but non optimal microscopic pattern and the numerical discretization error.
Let us emphasize that we do not expect that the resulting difference of compliance costs vanishes for the mesh size tending to zero.
In fact, we are interested in the remaining global modeling error and its associated spatial distribution.  Furthermore, we will use the
resulting error estimate to adapt the macroscopic finite element mesh.

Let $\lstrain$ be the solution of the continuous problem \eqref{eq:weak} involving the optimal elasticity tensor field $\ltensor$
resulting from an optimal sequentially laminated microstructure at each point. Likewise let $\sstrain_h$ be the discrete macroscopic solution
when using the two-scale model with piecewise constant effective tensors $\stensor_h$ obtained by solving the cell problems \eqref{eq:effectivevariational}
for a pattern of two rotated, orthogonal trusses.
Ultimately we are interested in an estimate of the corresponding difference of compliance cost values
\begin{equation*}
 \left| J[\ltensor;\lstrain] - J[\stensor_h;\sstrain_h] \right|,
\end{equation*}
in particular for $\stensor_h$ being the field of elasticity tensors resulting for the optimal choice of the microstructure parameters and for $\sstrain_h$ computed as the resulting
 discrete macroscopic strain in the finite element space $\Vh$.
To derive such an estimate  we employ the dual weighted residual approach \cite{BeKaRa00} for optimal control problems. Here, the controls---which have to reside in the same space---are
the coefficient functions of the elasticity tensors $\ltensor$, $\stensor_h$. Using the usual notation $\|\cdot\|_{m,p,A}$ for the $W^{{m,p}}$ Sobolev norm on a set $A$
we obtain the following theorem.
\begin{theorem}[Weighted a posteriori modeling error estimate] \label{theorem:estimates}
Given the continuous solution $\lstrain$ to \eqref{eq:weak} for an optimal effective elasticity tensor field $\ltensor$ (as obtained from optimal nested lamination) and the
macroscopic finite element solution $\sstrain_h$ for the two-scale model with a piecewise constant tensor field $\stensor_h$ we obtain for the difference of the associated
compliance cost values the estimate
\begin{equation}\label{eq:estimate}
 \left| J[\ltensor,\lstrain] - J[\stensor_h,\sstrain_h] \right| \leq \sum_E \eta_E(\lstrain,\ltensor,\sstrain_h,\stensor_h) + \mathcal{R} \,,
\end{equation}
where the cellwise  $\eta_E$  values are decomposed as follows
\begin{equation*}
\begin{aligned}
  \eta_E(\lstrain,\ltensor,\sstrain_h,\stensor_h) &:= \eta_E^{\strain} + \eta_{\partial E}^{\strain} + \fr12 \eta_E^{\etensor}  \quad \text{with} \\
  \eta_E^{\strain}              &:= \left|\int_{E} \div \left\{ \stensor_h \eps{\sstrain_h} \right\} \cdot \left(\lstrain-\sstrain_h\right) \d x \right| \,,  \\
  \eta_{\partial E}^{\strain}   &:= \left|\int_{\partial E} j (\stensor_h \eps{\sstrain_h}) \cdot \left(\lstrain-\sstrain_h\right) \d a(x) \right| \,,  \\
  \eta_E^{\etensor}             &:= \left|\int_{E}  ( \ltensor - \stensor_h ) \, \eps{\sstrain_h} : \eps{\sstrain_h} \d x\right| \,,
 \end{aligned}
\end{equation*}
and $j$ denotes the jump of the normal stress across an edge. Furthermore, the remainder is given by
\[
\mathcal{R} := \frac1{2}a(e_\etensor; e_\strain,e_\strain)
\]
and is thus of higher order in the difference of states
$e_\strain := \lstrain -\sstrain_h$ and the difference of elasticity tensors $e_\etensor := \ltensor - \stensor_h$.
\end{theorem}
\begin{proof}
The proof follows the usual procedure for deriving dual weighted residual estimates (cf. \cite{BeKaRa00} for the method in general or \cite{GeRu15}
for a shape optimization problem based on microscopic sequential lamination).
We consider the Lagrangian
\begin{equation}
\label{eq:L}
\tilde \Lag(\etensor,\strain,p) := J[\etensor,\strain] + a(\etensor;\strain,p) - l(p)
\end{equation}
for arbitrary $\etensor$, $\strain$, and $p$.
The newly introduced function $p$ represents the solution to an adjoint equation. In case of the compliance cost functional it is, however,
trivially determined by $p=-u$, cf. \cite{Al02}.
According to our assumption $\lstrain$ and $\sstrain_h$ solve the continuous and the discrete state equation for given
$\ltensor$ and $\stensor_h$, respectively. Thus, we obtain $a(\ltensor;\lstrain,-\lstrain) - l(-\lstrain) = 0$,  $a(\stensor_h;\sstrain_h,-\sstrain_h) - l(-\sstrain_h) = 0$, which implies
\begin{equation}\label{eq:errorlagrangian}
 e_\Lag := \tilde \Lag(\ltensor,\lstrain,-\lstrain)- \tilde \Lag(\stensor_h,\sstrain_h,-\sstrain_h)
 = J[\ltensor,\lstrain]-J[\stensor_h,\sstrain_h]\,.
\end{equation}
Thus, we focus on the error in the Lagrangian involving the tensor fields
$\ltensor$, $\stensor_h$ and the corresponding primal solutions $\lstrain$, $\sstrain_h$.
We use the shortcut notation $\Lag(\etensor,\strain) := \tilde \Lag(\etensor,\strain,-\strain)$. \\
To derive a first order expansion of $e_\Lag$ we interpolate linearly between the quantities of the continuous lamination and the discrete two-scale problem
and obtain for the difference in the Lagrangian
\begin{equation*}
 e_\Lag = \int_0^1 \tdiff{s} \Lag (\stensor_h + s e_\etensor, \sstrain_h + s e_\strain ) \,\mathrm{d}s\,.
\end{equation*}
We define $f(s)$ as the above integrand, \ie\ $f(s) := \tdiff{s} \Lag (\stensor_h + s e_\etensor, \sstrain_h + s e_\strain)$,
and apply the trapezoidal rule $\int_0^1 f(s) \mathrm{d}s  =  \frac12 ( f(0) + f(1) ) - \frac12 \int_0^1 f''(s) s(1-s)\, \mathrm{d}s$.
All derivatives exist and can be explicitly calculated, since $\Lag$ is a polynomial in its arguments, but it is convenient to do this
only at the final stage of the computation.
From the assumption that $\ltensor$ is the optimal field of elasticity tensors and $\lstrain$ the associated primal elastic solution
we deduce that $(\ltensor,\lstrain, - \lstrain)$ is a stationary point of the Lagrangian.  Thus, we obtain that
\[
f(1) = \nabla \Lag ( \ltensor,\lstrain ) \cdot (e_\etensor,e_\strain)^\top = 0\,.
\]
Hence, using the remainder term
$\mathcal{R} := -\frac12 \int_0^1 \frac{\mathsf{d}^3}{\mathsf{d} s^3}
\Lag (\stensor_h + s e_\etensor, \sstrain_h + s e_\strain) \, s \, (1-s) \mathrm{d}s$
we end up with the following representation of the difference of compliance cost values:
\begin{equation}\label{eq:errorrep}
  e_\Lag =  \frac12 \Lag_{,\strain}  (\stensor_h, \sstrain_h) ( \lstrain - \sstrain_h )
          + \frac12 \Lag_{,\etensor} (\stensor_h, \sstrain_h) ( \ltensor - \stensor_h )
          + \mathcal{R} .
\end{equation}
Next, we decompose the first two terms into contributions on the elements and element boundaries of the macroscopic finite element mesh and obtain using integration by parts
\begin{align}
 &  \Lag_{,\strain}(\stensor_h,\sstrain_h) (\lstrain-\sstrain_h) \notag \\
 &= -2 \int_\workdom \stensor_h \, \eps{\sstrain_h} : \eps{\lstrain-\sstrain_h} \d x
    +2 \intga g \cdot (\lstrain-\sstrain_h) \d a(x) \notag \\
 &=  2 \, \sum_E \Big(
     \int_{E} \div \left\{ \stensor_h \eps{\sstrain_h} \right\} \cdot (\lstrain - \sstrain_h) \d x
    -\int_{\partial E} \stensor_h \eps{\sstrain_h} n \cdot (\lstrain-\sstrain_h) \d a(x) \notag  \\
 &   \qquad \qquad + \int_{\partial E \cap \Gamma_N} g \cdot (\lstrain-\sstrain_h) \d a(x) \Big) \notag  \\
 & \leq2 \sum_E \left( \eta_E^{\strain} + \eta_{\partial E}^{\strain} \right)  \label{eq:fullprimal}
\end{align}
with the  postulated residual terms
$\eta_E^{\strain} = |\int_{E} \div \left\{ \stensor_h \eps{\sstrain_h} \right\} \cdot \left(\lstrain-\sstrain_h\right) \d x |$ and
$\eta_{\partial E}^{\strain} = |\int_{\partial E} j (\stensor_h \eps{\sstrain_h}) \cdot$ $\left(\lstrain-\sstrain_h\right) \d a(x) |$.
Thereby,
$j(\stress)(x) = \frac12 \left[ \stress(x) \cdot \normal(x) \right]$ on interior edges, $j(\stress)(x) = \stress(x) \cdot \normal(x) - g(x)$ on $\Gamma_N$, and
$j(\stress)(x) = 0$ on $\Gamma_D$. Here,
$\left[ \stress(x) \cdot \normal(x) \right]$ denotes the jump of the normal stress across an edge, \ie\
$
\left[ \stress(x) \cdot \normal(x) \right] = \left( \left. \stress(x)  \right|_{E} - \left. \stress(x)  \right|_{E'}\right)\cdot \normal(x)
$
for $x \in E \cap E'$,
and $\normal$ being the normal on $E \cap E'$ pointing from $E$ to $E'$.
For the second term in \eqref{eq:errorrep}  we obtain the straightforward splitting
\begin{equation}
  \Lag_{,\etensor} (\stensor_h, \sstrain_h) ( \ltensor - \stensor_h )
=- \int_{\workdom} ( \ltensor - \stensor_h ) \, \eps{\sstrain_h} : \eps{\sstrain_h} \d x \le
\sum_{E}  \eta_E^{\etensor}\,.
\label{eq:fullcontrol}
\end{equation}
Finally, $\frac{\mathsf{d}^3}{\mathsf{d} s^3}
\Lag (\stensor_h + s e_\etensor, \sstrain_h + s e_\strain)=-6a(e_\etensor; e_\strain, e_\strain)$
does not depend on $s$, and a straightforward integration concludes the proof.
\end{proof}
\medskip

Let us remark that the first two terms in the definition of $\eta_E$  (\ie \; $\eta_E^{\strain} + \eta_{\partial E}^{\strain}$)
measure the discretization error. Indeed, $\eta_E^{\strain}$ and $\eta_{\partial E}^{\strain}$ include the element and the singular residual of the primal problem along with weighting terms.
They both are expected to vanish for the mesh size tending to $0$.  The last term represents the difference of the stored elastic energy
$\mathcal{E}[\etensor, \strain] = \tfrac12 \int_{\workdom} C \, \eps{\strain} : \eps{\strain} \d x$ for the different elasticity tensors $\ltensor$ and $\stensor_h$ evaluated for the discrete strain $\sstrain_h$.
This term measures the modeling error and is not required to vanish in general.
It measures the error caused by the choice of the rotated truss microstructure compared to the optimal nested laminate.

In  Theorem \ref{theorem:estimates}  we have restricted ourselves to an optimization of the compliance cost. A generalization to other cost functionals, such as a tracking type
cost functional, is straightforward. However, in this case the dual solution $p$ no longer coincides with $-u$. Furthermore, the Lagrangian is not necessarily a polynomial  in its arguments and thus there is in general no explicit representation of the residual term $\mathcal{R}$  as stated in the theorem.

\section{Derivation of effective local error indicators} \label{sec:weighting}
The cellwise values $\eta_E$ of the weighted a posteriori error estimate in Theorem \ref{theorem:estimates}
depend on the computed numerical solution $\sstrainh$ and on the unknown, continuous displacement field $\lstrain$ corresponding to the as well unknown, optimal elasticity tensor field $\ltensor$. Furthermore, the energy difference terms $\eta_E^C$ involve the optimal continuous elasticity tensor $\ltensor$.
In order to make practical usage of this estimate we need to compute
suitable approximations of the weighting terms  $\left(\lstrain-\sstrain_h\right)$
and to efficiently estimate  the locally optimal effective elasticity tensor $\ltensor$. Based on these approximations we then
replace $\eta_E$ by a computable approximation. This can then in the fully practical algorithm be used to steer the grid refinement and to
 evaluate a distribution of the modeling error due to the replacement of the nested lamination microstructure by the
 microstructure consisting of rotated, orthogonal trusses.

\emph{Approximation of $\lstrain$.}
On a given mesh $\mathcal{M}_h$ let us consider a two-scale model, where the effective discrete macroscopic strain $u_h \in \Vh$
in the Galerkin approximation of \eqref{eq:weak}
on  $\mathcal{M}_h$ results from a microscopic nested laminate construction.  We denote by $\ltensorh$ and $\lstrainh$ the optimal effective elasticity tensor and strain resulting from a minimization of the
compliance cost functional. Thereby, we take into account discrete tensor fields which are piecewise constant on the elements of the mesh.

It is well-known
that the $\ltensorh$ can be easily retrieved from the corresponding
stress field $\sigma_h=\ltensorh \epsilon[\lstrainh]$.
Indeed, the lamination directions coincide with the eigendirections of $\sigma_h$ and based on this insight the optimal ratios between the two materials in each involved lamination of the
local composite can be easily identified as functions of the eigenvalues of the stress $\sigma_h$ (for details we refer to \cite{Al02}). Altogether, we obtain
\begin{equation}\label{eq:C}
\ltensorh = \mathbf{C}(\alpha(\sigma_h),\lambda_1(\sigma_h),\lambda_2(\sigma_h))
\end{equation}
for some function
\begin{equation*}
 \mathbf{C}:  \R^3 \to \R^{2^4};\, \alpha,\lambda_1,\lambda_2  \mapsto \mathbf{C}(\alpha,\lambda_1,\lambda_2)\,,
\end{equation*}
which maps  the rotation angle $\alpha$ from the canonical basis into the basis of the eigendirections of the stress
and the two eigenstresses $\lambda_1$ and $\lambda_2$ to
the effective elasticity tensor of the optimal nested laminate associated with the corresponding underlying elastic stress. For the detailed derivation of this function
we refer to \cite{Al02}. In Section \ref{sec:impl} we give the explicit formulas.
Equation \eqref{eq:C} gives rise to a simple iterative minimization algorithm starting from some initial strain. Indeed, given $u^\mathrm{L}_{h,i-1}$
one first evaluates
\begin{equation*}
C_{h,i}^\mathrm{L}=\mathbf{C}(\alpha(\sigma_{h,i}),\lambda_1(\sigma_{h,i}),\lambda_2(\sigma_{h,i}))
\end{equation*}
with
$\sigma_{h,i}=C^\mathrm{L}_{h,i-1} \epsilon[u^\mathcal{L}_{h,i-1}]$
and then computes
$u^\mathrm{L}_{h,i}$ as the discrete solution of the Galerkin approximation of \eqref{eq:weak} with the elasticity tensor $C_{h,i}^\mathrm{L}$ in $\Vh$.
As a suitable initial elasticity tensor and strain field we consider $\stensorh$ and $\sstrainh$, respectively.
Then, we compute for some fixed $k\in \mathbb{N}$ the elasticity tensor $\ltensorhk$ and the associated strain $\lstrainhk$.

Following the usual procedure in the context of the dual weighted error estimation approach we now use a higher order interpolation of a discrete PDE solution in a neighborhood of each cell as a higher order approximation for the continuous PDE solution and apply this to $\lstrainhk$. Specifically, for a cell $E$ of an adaptive mesh $\mathcal{M}_h$ we proceed as follows. Let us assume that $\mathcal{M}_h$ is generated based on an adaptive quadtree data structure and that $E$ is one of the four child cells of a coarser cell $E^c$. Now, we consider the Lagrangian interpolation $\mathcal{I}^{(4)}_h$ on the space of bi-quartic polynomials based on functional evaluation on the union of bi-quadratic Lagrangian nodes on the children of $E^c$. Given   $\lstrainhk$ we define  $\tlstrainhk$ on the cell $E$ as
$\mathcal{I}^{(4)}_h \lstrainhk |_E\,$.

With $\tlstrainhk$ at hand, we can define the approximations
\begin{equation}\label{eq:approxweight}
\tilde \eta_E^{\strain}            = \left|\int_{E} \div \left\{ \stensor_h \eps{\sstrain_h} \right\} \cdot \left(\tlstrainhk-\sstrain_h\right) \d x \right|, \quad
\tilde \eta_{\partial E}^{\strain} = \left|\int_{\partial E} j (\stensor_h \eps{\sstrain_h})          \cdot \left(\tlstrainhk-\sstrain_h\right) \d a(x) \right|,
\end{equation}
which can be computed based on a Gaussian (tensor product) quadrature with order $5$ and $3\times 3$ nodes, respectively.

\emph{Approximation of $\ltensor$.} The stress tensor $\sigma$ is uniquely determined by the rotation angle $\alpha$ and its two eigenvalues $\lambda_1$ and $\lambda_2$, \ie
\[
\sigma = R(\alpha)  \left( \begin{array}{cc}
      \lambda_1 & 0 \\
      0 & \lambda_2 \\
   \end{array} \right) R(\alpha)^T \,,
\]
where $R(\alpha)$ is the rotation from the canonical basis into the basis of the eigenstrains.
On the other hand $\sigma = \ltensor \epsilon[\lstrain]$  with $\ltensor = \mathbf{C}(\alpha,\lambda_1,\lambda_2)$.
Hence, we define a function
\begin{equation}\label{eq:newton}
F(\alpha,\lambda_1,\lambda_2) = \mathbf{C}(\alpha,\lambda_1,\lambda_2) \epsilon[\lstrain] -
R(\alpha)  \left( \begin{array}{cc}
      \lambda_1 & 0 \\
      0 & \lambda_2 \\
   \end{array} \right) R(\alpha)^T\,,
\end{equation}
which maps the three parameters $\alpha, \, \lambda_1, \, \lambda_2$ to a symmetric $2\times 2$ matrix with its three degrees of freedom.
Roots of this function correspond  to stresses $\sigma$ and elasticity tensors $\ltensor$
for a given strain tensor $\epsilon[\lstrain]$ in the optimal laminate configuration.
In the implementation we use Newton's method to compute for given $\epsilon[\tlstrainhk]$ a root
$(\alpha, \lambda_1, \lambda_2)$ of $F$. Thus, we obtain
$\tltensorhk = \mathbf{C}(\alpha,\lambda_1,\lambda_2)$ as an admissible elasticity tensor corresponding  to an optimal nested laminate microstructure
for a given approximation  $\tlstrainhk$ of the optimal strain.
Finally, based on  $\tltensorhk$ we compute as an approximation of the local modeling error term $\eta_E^{\etensor}$
\begin{equation}\label{eq:approxC}
\tilde \eta_E^{\etensor} =  \int_{E} ( \tltensorhk - \stensor_h ) \, \eps{\sstrain_h} : \eps{\sstrain_h} \d x
\end{equation}
applying  the above  Gaussian (tensor product) quadrature of order $5$.

\emph{Approximate evaluation of $\eta_E$.}
Taking into account the above approximation of the local weighting terms in $\tilde \eta_E^{\strain}$, $\tilde \eta_{\partial E}^{\strain}$
and the local modeling error term $\tilde \eta_E^{\etensor}$, we obtain as an approximate upper bound for the
difference $J[\ltensor,\lstrain] - J[\stensor_h,\sstrain_h]$ in \eqref{eq:estimate} the term
$\sum_E \tilde \eta_E(\sstrain_h,\stensor_h) + \mathcal{R}$ with
\begin{equation}
\tilde \eta_E(\sstrain_h,\stensor_h) =  \tilde \eta_E^{\strain} + \tilde \eta_{\partial E}^{\strain} + \fr12 \tilde \eta_E^{\etensor}\,,
\end{equation}
which is evaluated based on a given pair of discrete strain $\sstrain_h$ and discrete elasticity tensor $\stensor_h$, which correspond to the optimal compliance cost in the case of the microstructure formed by rotated, orthogonal trusses of varying width and rotation angle.

In what follows, we give
some numerical evidence that an early truncation of the laminates algorithm already yields a sufficiently good approximation to the fully converged
solution of the lamination model. In Figure \ref{fig:lamApprx} we show the absolute error in the compliance and the $L^2$-error in the elastic solution for each
iteration of the alternating lamination algorithm, see below,  starting from the two-scale solution $\sstrainh$ corresponding to $\stensorh$. The difference is computed \wrt\
the final value at convergence, \ie\ when subsequent compliance values differ by no more than $10^{-8}$, of the lamination algorithm starting from the same initial values.
In fact, we will use $k=50$ iterations for our numerical computations.

\begin{figure}[!ht]
\hfill
\includegraphics[width=.6\linewidth]{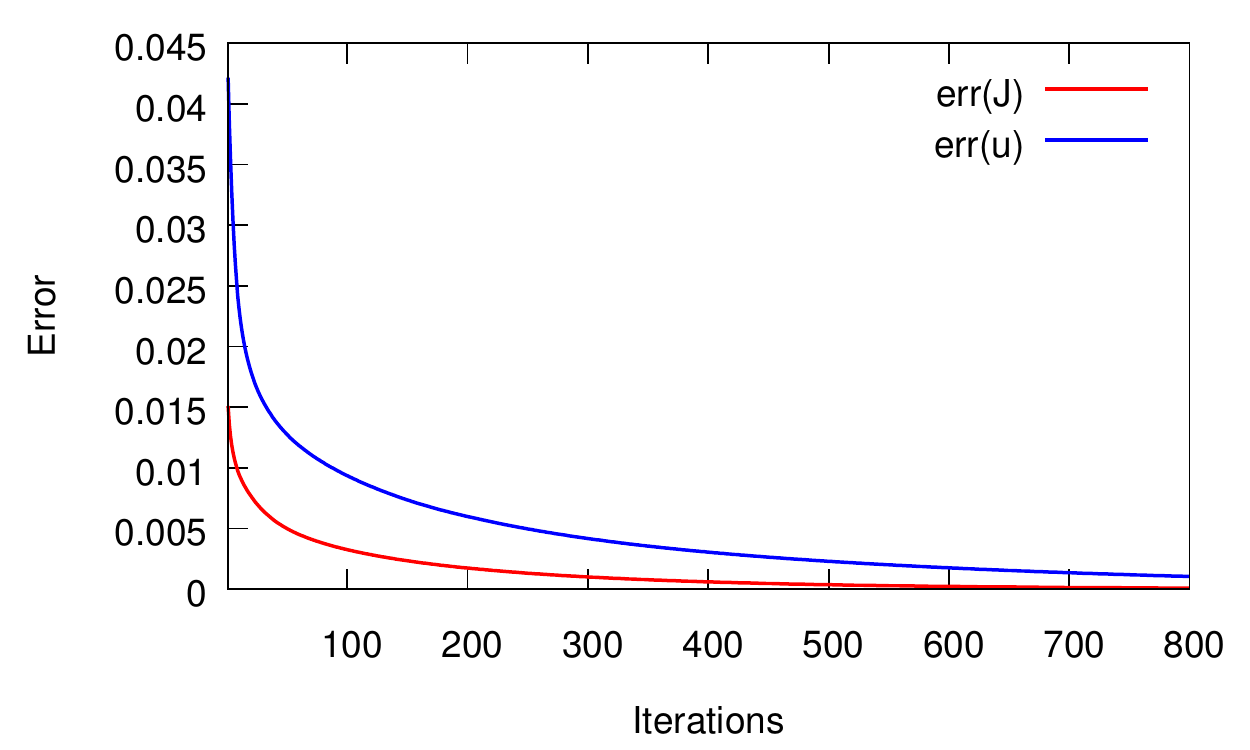}
\hfill\mbox{}
\caption{After each iteration of the laminates algorithm, starting from an initial state given by the two-scale model, the compliance objective and elastic solution
are compared to the final state. The error in the compliance is the absolute value, the error in the elastic solution is the $L^2$-error.}
\label{fig:lamApprx}
\end{figure}

\section{Implementation} \label{sec:impl}
Our macroscopic computations are performed on an regular quadrilateral mesh implemented within the \texttt{QuocMesh}
library\footnote{{http://numod.ins.uni-bonn.de/software/quocmesh}}.
The library also provides a collocation type boundary element method for the solution of cell problems on the microscale.
Furthermore it includes a classical Newton scheme using optimal step size control to solve \eqref{eq:newton}.
Adaptive refinements are realized via uniform subdivision and handling of constrained hanging nodes.
Checkerboard instabilities were reported for a density optimization model by Jog and Haber \cite{JoHa96}
and by Jouve and Bonnetier \cite{BoJo98} for the sequential lamination microstructure.
Following the observations reported in these papers we
use bi-quadratic finite elements for stabilization on the macroscale.
For numerical integration we use a
Gauss quadrature rule of order 5 which turned out to be sufficient. \\
We reimplemented the alternating algorithm for sequential lamination microstructures suggested in \cite{AlBoFr97} and already
described in \cite{GeRu15}.

Given the rotation $\alpha$ of the dominant eigenvector of $\sigma_h$ and its eigenvalues $\lambda_1$ and $\lambda_2$, the lamination
parameters and the effective elasticity tensor in equation \eqref{eq:C} are explicitly given by
\begin{equation*}
\begin{aligned}
m[\stress_h]      &= \frac{|\lambda_2(\stress_h)|}{|\lambda_1(\stress_h)|+|\lambda_2(\stress_h)|}, \;
\theta[\stress_h] = \min\left\{ 1, \sqrt{ \frac{2 \mu\! +\! \lambda}{4 \mu (\mu\!+\!\lambda) \, l} } ( |\lambda_1(\stress_h)|+|\lambda_2(\stress_h)| ) \right\} \,,\\[6pt]
C^\ast_{mnop}[q] &= R[\alpha] \, \bar{C}[m,\theta] := Q_{mi}[\alpha] \, Q_{nj}[\alpha] \, Q_{ok}[\alpha] \, Q_{pl}[\alpha] \; \bar{C}_{ijkl}[m,\theta] \,, \\[6pt]
\bar{C}_{1111}[m,\theta] &= \frac{ 4 \kappa \mu (\kappa\!+\!\mu) \theta (1\!-\!\theta (1-m))(1\!-\!m) } {4 \kappa \mu m (1\!-\!m) \theta^2 + (\kappa\!+\!\mu)^2 (1\!-\!\theta)}, \;
\bar{C}_{2222}[m,\theta] = \frac{ 4 \kappa \mu (\kappa\!+\!\mu) \theta (1\!-\!\theta m)m }         {4 \kappa \mu m (1 \!-\!m) \theta^2 + (\kappa\!+\!\mu)^2 (1\!-\!\theta)} \,, \\
\bar{C}_{1122}[m,\theta] &= \frac{ 4 \kappa \mu \lambda \theta^2 m (1\!-\!m) }                  {4 \kappa \mu m (1\!-\!m) \theta^2 + (\kappa\!+\!\mu)^2 (1\!-\!\theta)} \,.
\end{aligned}
\end{equation*}
Here $R$ is a linear mapping rotating the tensor $\bar{C}$ given in reference configuration into the appropriate coordinate frame given by rotation parameter $\alpha$.
In its definition $Q$ are $2 \times 2$ rotation matrices and the Einstein summation convention is used. The bulk modulus is defined as $\kappa  = \lambda + \mu$
and $l$ is a Lagrangian multiplier used in the alternating algorithm to enforce the volume constraint.
The tensor $\bar{C}$ is complemented using the symmetry relations and filling the remaining entries with  $0$. This yields a singular elasticity tensor that has to
be regularized by adding a small constant.

Within each cell the underlying microstructure is specified by a small set of parameters leading to a finite
dimensional constrained  optimization problem solved by the open source software \texttt{Ipopt} \cite{Wa02,WaBi06}.
It implements an SQP type minimization scheme using limited memory BFGS updates to approximate the Hessian.
To steer the adaptivity we follow a D\"orfler strategy \cite{Do96} marking cells giving rise to the top $40\%$ of the total
estimated error.

\section{Numerical results} \label{sec:results}
In the following we discuss the concrete application of our adaptive algorithm to  four textbook examples of 2D shape optimization problems.
For each of them we  visualize the resulting solution and refinement patterns and list the error indicator values and the computed cost values
at each refinement step.
As expected, the modeling error cannot be reduced beyond a problem-dependent positive lower bound.
In fact, after a substantial reduction of the discretization and the modeling error in the initial refinement stages, the algorithms builds an oscillatory pattern on the grid scale in
certain regions. We discuss below the implications of these observations for the appropriate usage of the a posteriori error control.

The first scenario is a carrier plate modeled by a square domain $\workdom = [0,1]^2$ fixed at the bottom and subject to a uniform shearing load
on the top, cf.~Figure~\ref{fig:res_carrier}.
The next scenario is a cantilever, cf.~Figure~\ref{fig:res_cantilever}. It is modeled by a rectangular domain $\workdom = [0,1] \times [0,0{.}5]$,
fixed at the left hand side and subject to a downwards pointing load located in the center of the right hand side.
The third example is a bridge configuration given on the domain $\workdom = [0,1] \times [0,0{.}5]$, cf.~Figure~\ref{fig:res_bridge}.
We prescribe roller boundary conditions on a small fraction of the lower boundary on the left and right hand side,
\ie\ only the vertical displacement component is kept fixed there. In between a uniform downwards pointing load is applied.
Finally we consider an L-shaped domain $\workdom = [0,1]^2 \setminus [0{.}5,1]^2$
fixed on top and subject to a downwards pointing load in the center of the lower right boundary, cf.~Figure~\ref{fig:res_ldomain}.
All applied loadings have a magnitude of $1$. For the cantilever scenario the global volume fraction is constrained to $50\%$,
for all other scenarios to $67\%$, respectively.

For the carrier plate scenario we show the macroscopic computational domain at several steps of the adaptive scheme
in Figure~\ref{fig:res_carrier}.
Moreover the optimized microscopic geometries underlying each macroscopic element of the grid,
the associated elementwise volume densities,
and a color coding of the von Mises stress are shown.
For the other scenarios we depict grid, visualization and von Mises stress of two intermediate refinement steps in Figures
\ref{fig:res_cantilever}, \ref{fig:res_bridge}, \ref{fig:res_ldomain}.
To analyze the computed error indicators leading to the refined grids we list the contributions of each term for every refinement
step in Tables \ref{fig:res_carrier_error} and \ref{fig:res_other_error}.
A striking observation is that the error indicator does not decrease at later stages of the adaptive algorithm. As noted earlier one cannot expect the
indicator to go to zero due to the non vanishing modeling error. However, the tables show that while the modeling error increases at most mildly both
discretization error terms show significant growth after a few refinement steps.
In fact, this appears in regions where the stress tensor would indicate a nested laminate construction as the optimal local pattern.
Thus, this locally optimal two-scale geometry cannot be realized by the rotated truss construction.
As a compensation of this deficiency the algorithm seems to try to establish an additional intermediate scale on the grid level,
This does not appear to be a  classical numerical instability caused by a non appropriate numerical discretization ansatz but
a laminate type oscillating pattern as a consequence of the non optimality of the choosen microscopic model.
 This can be observed in the images,
as for example in the last row of Figure \ref{fig:res_carrier},
which shows a color coding of the local discretization error contributions.
These oscillations have small effect on the energy but
lead to significant additional discretization errors and therefore to the observed increase of the global discretization error estimate.
Furthermore, these oscillations have an impact on the  optimal compliance cost computed numerically on the various grids.
This optimal cost depends on the discrete solution of the underlying constrained optimization problem and is characterized by a discrete saddle point of the
Lagrangian \eqref{eq:L}.
In particular, it increases if the elastic problem is not fully resolved, which is the case if the coefficients
are rapidly varying. These  discretization errors
give rise to a substantial increase in the later refinement stages, as can be seen in the tables.
The process seems to be self-propelling and propagating during subsequent refinement steps, resulting in
refined computational domains that seem uneligible for the considered scenarios.

As a comparison we show a refined grid after $24$ refinement steps for the carrier plate scenario using
the sequential lamination model from the earlier work \cite{GeRu15} in Figure~\ref{fig:res_carrier_laminates}. Here one observes sharply
resolved interfaces between regions of diverse material density. However, in the central gray domain where the optimal pattern is an actually nested laminate
the grade of refinement keeps being moderate.
The observed phenomenon does not stem from the adaptivity of the grid, as illustrated by the
corresponding results based on four uniform refinement steps for
the cantilever scenario, starting from the same initial grid as before.
In Table \ref{fig:res_carrier_erroruni} we again recognize the dominant
increase of the discretization error terms.
There is no a priori error analysis of the modeling error and the discretization error at hand which would
allow to study this type of phenomena in full depth. Nevertheless in the context of a posteriori error estimation these observations lead to a practical strategy for the use of
the error estimate as a stopping criterium in the refinement process to avoid overrefinement.

In summary, our results illustrate good convergence of the objective functional and of the
general structure of the two-scale shape. They also show that the modeling error can be reduced via the adaptive meshing
strategy until a problem-dependent lower bound is approached.
At this point the limit of the chosen microscopic model of the rotated trusses for the local shape pattern is obviously reached, as indicated by the necessity
to construct an intermediate pattern via oscillations on the grid scale.  The proposed a posteriori error bounds show a clear indication of this effect
and allow to stop the algorithm at that point. Extreme refinement appears to be neither necessary nor useful in the present setting.
Without an additional regularity term in the cost functional which penalizes strong variations in the parameters of the
microscopic patterns, the cost functional cannot be further reduced via a refinement of the grid, as observed in the tables.
We also observe  a clear coincidence of the increase in discretization error with the increase in the numerically computed cost value.

%
%

\begin{figure}[!ht]
\includegraphics[width=.19\linewidth]{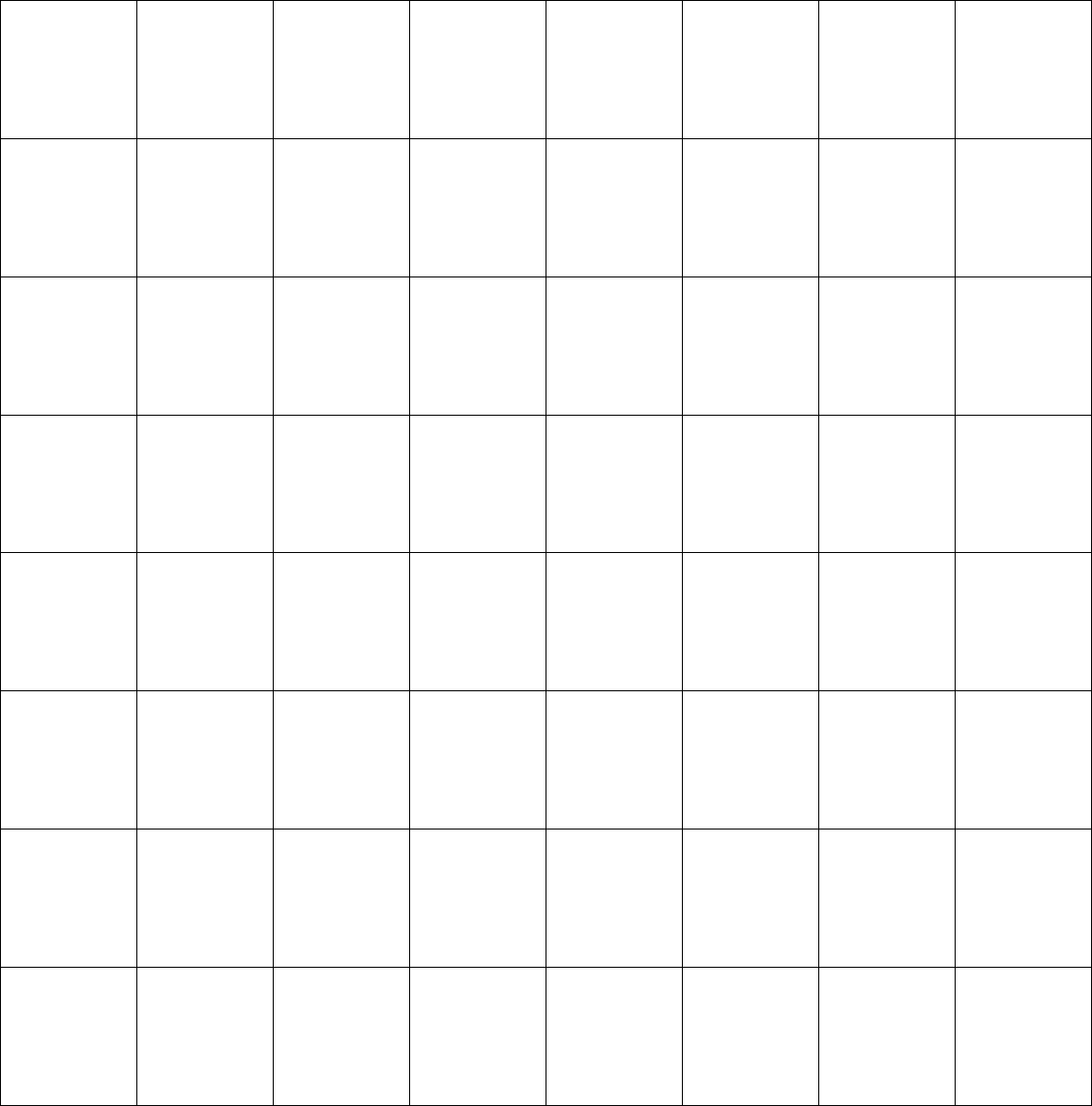}
\hfill
\includegraphics[width=.19\linewidth]{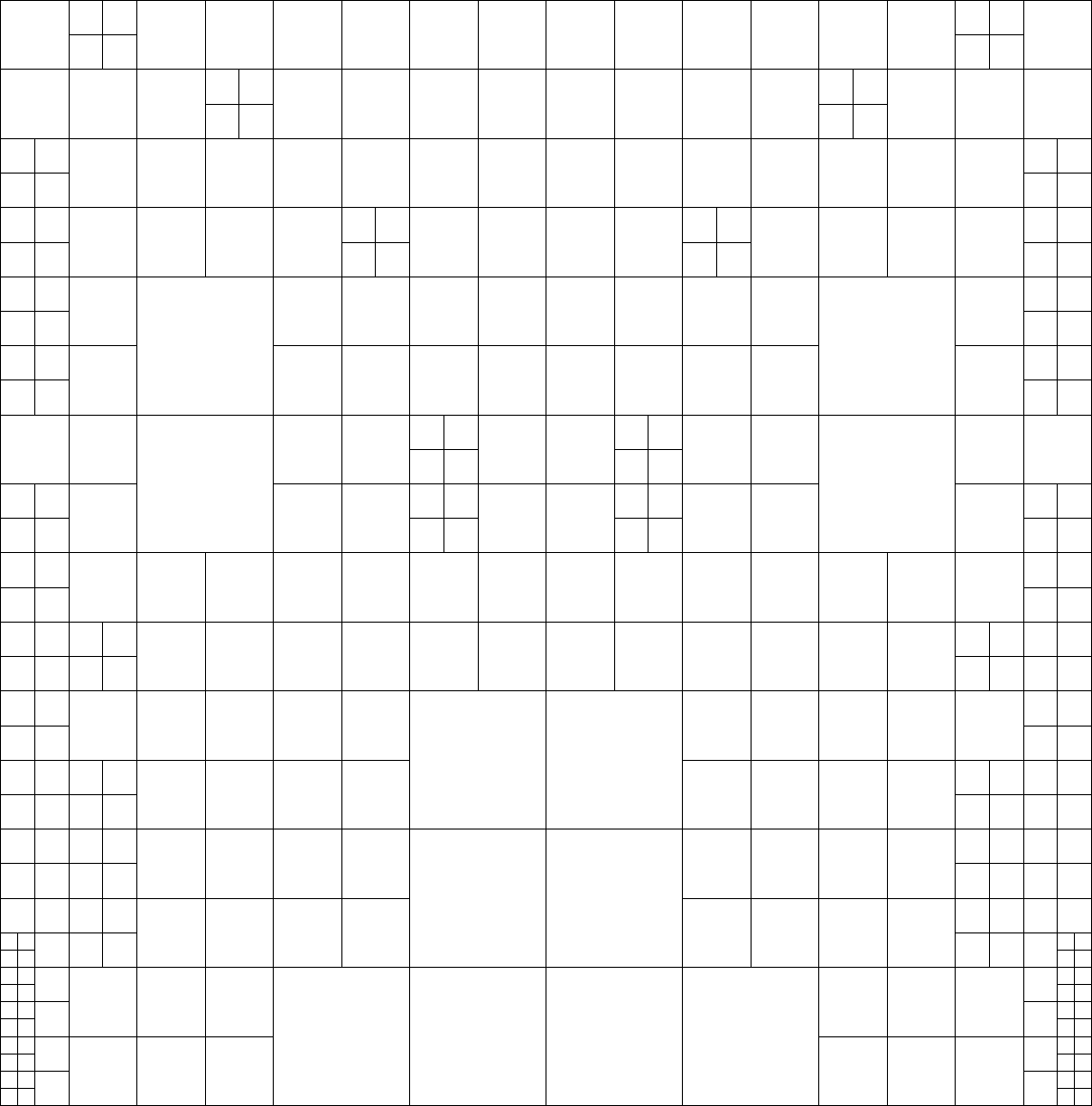}
\hfill
\includegraphics[width=.19\linewidth]{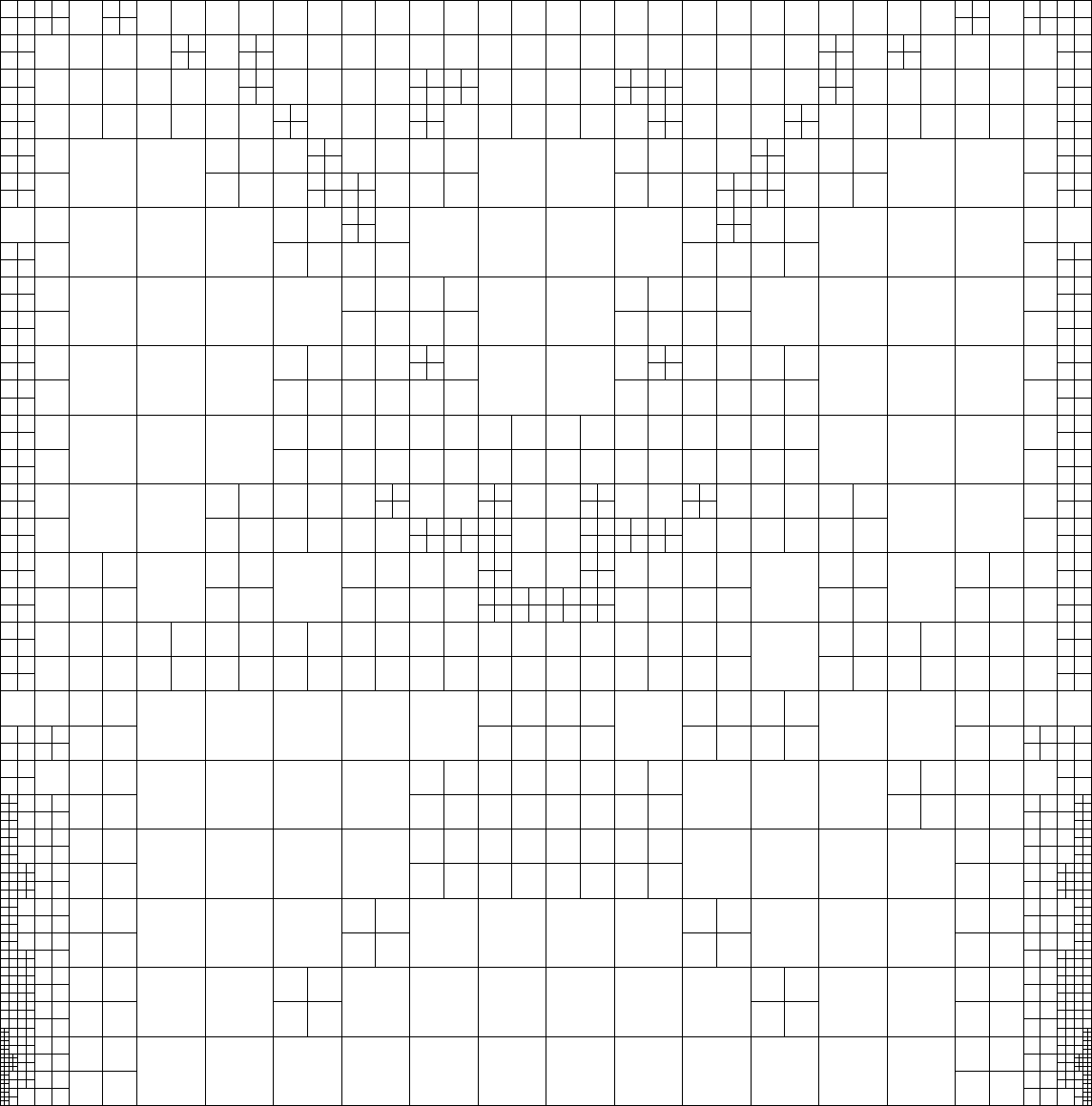}
\hfill
\includegraphics[width=.19\linewidth]{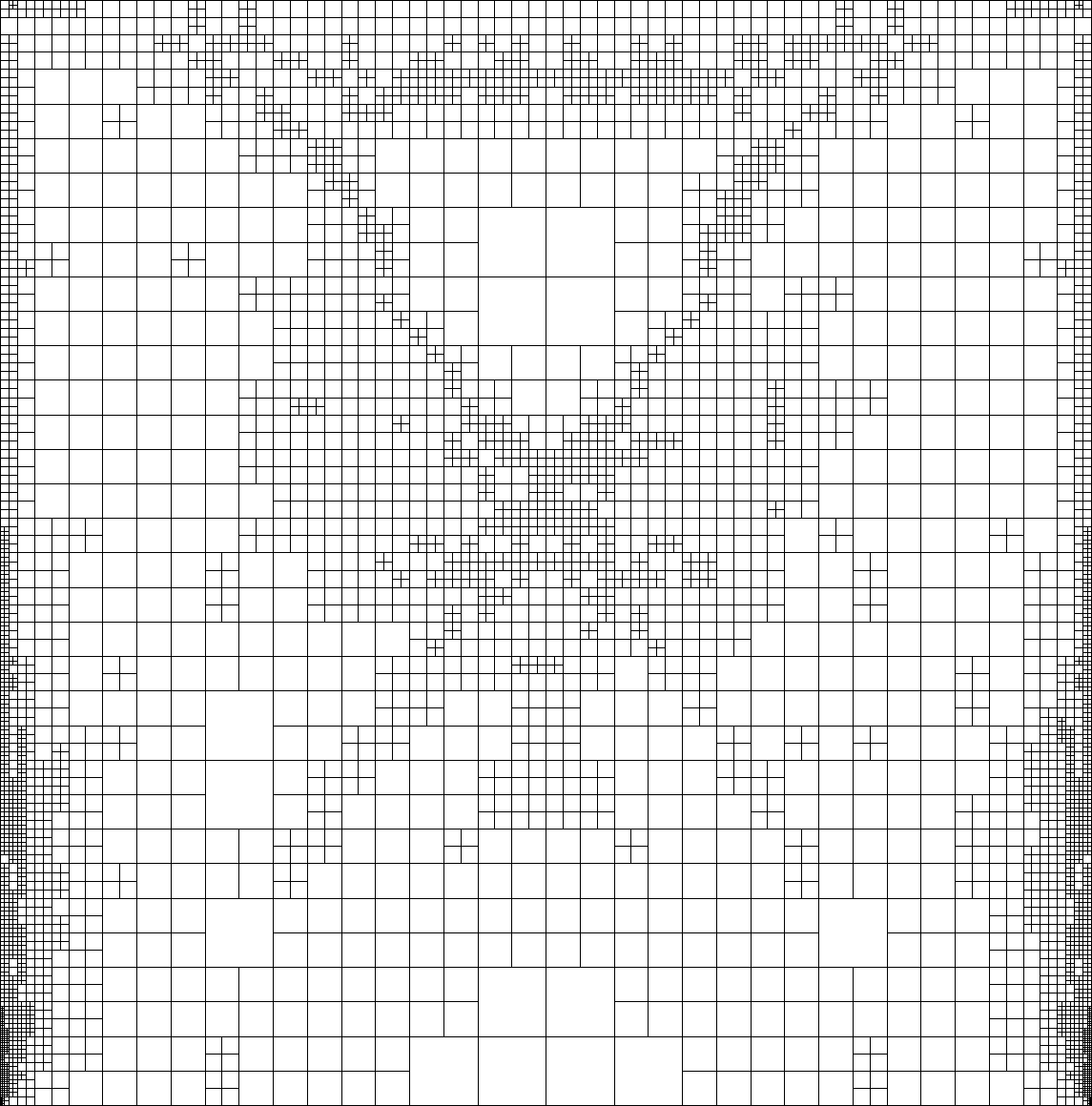}
\hfill
\includegraphics[width=.19\linewidth]{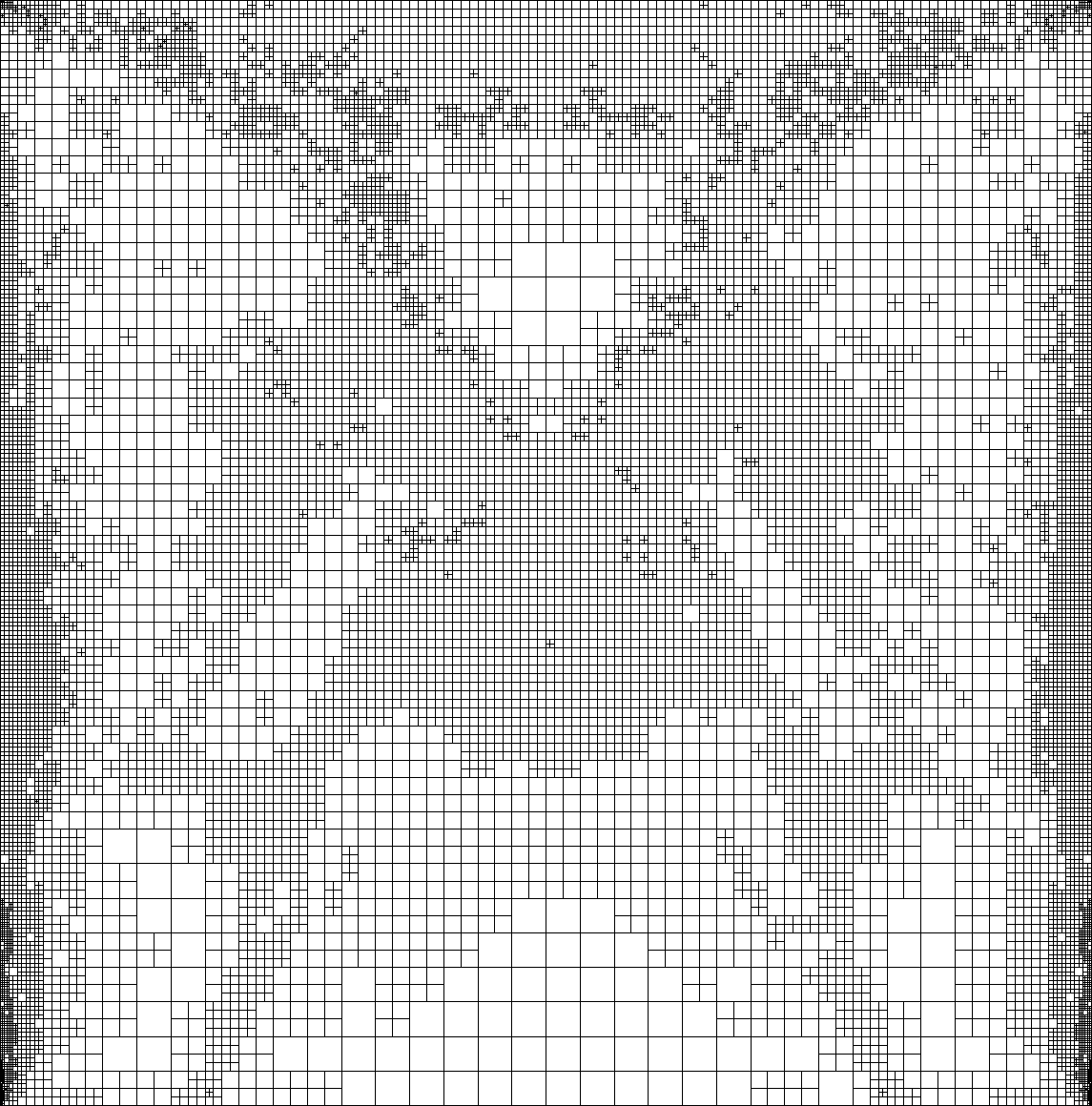} \\ \smallskip

\includegraphics[width=.19\linewidth]{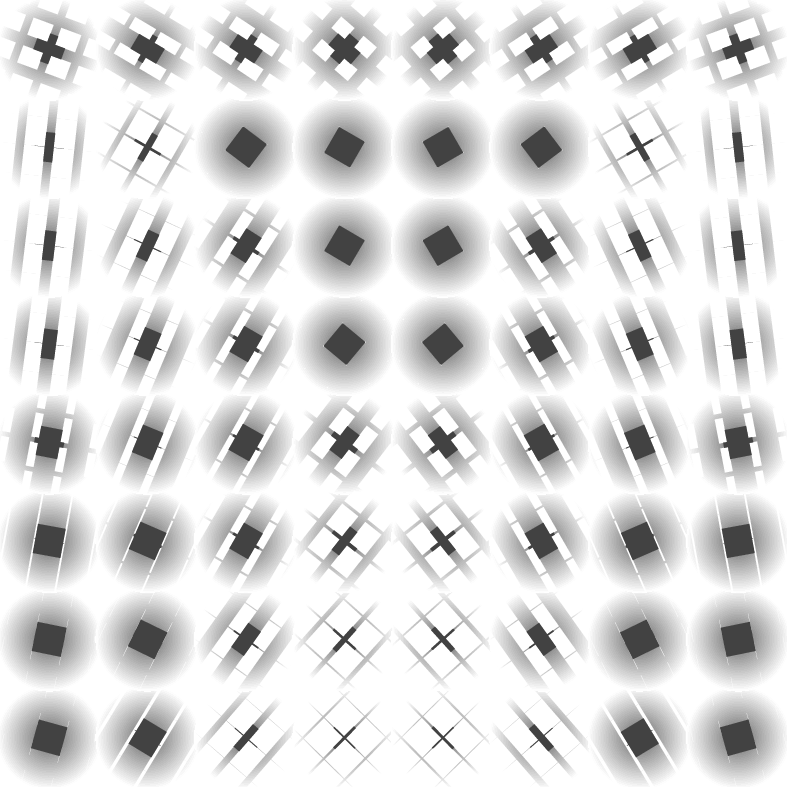}
\hfill
\includegraphics[width=.19\linewidth]{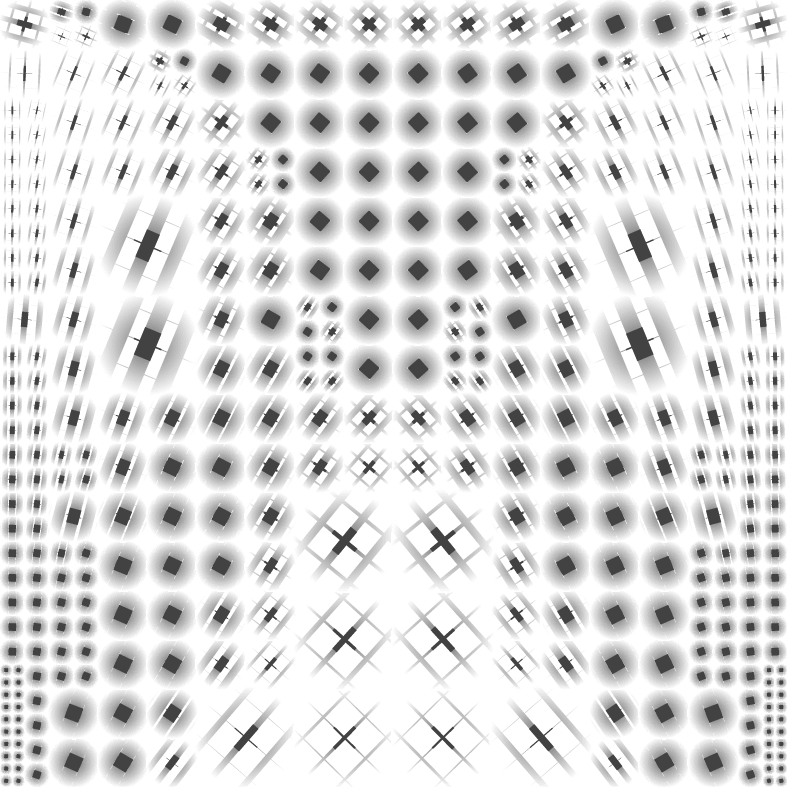}
\hfill
\includegraphics[width=.19\linewidth]{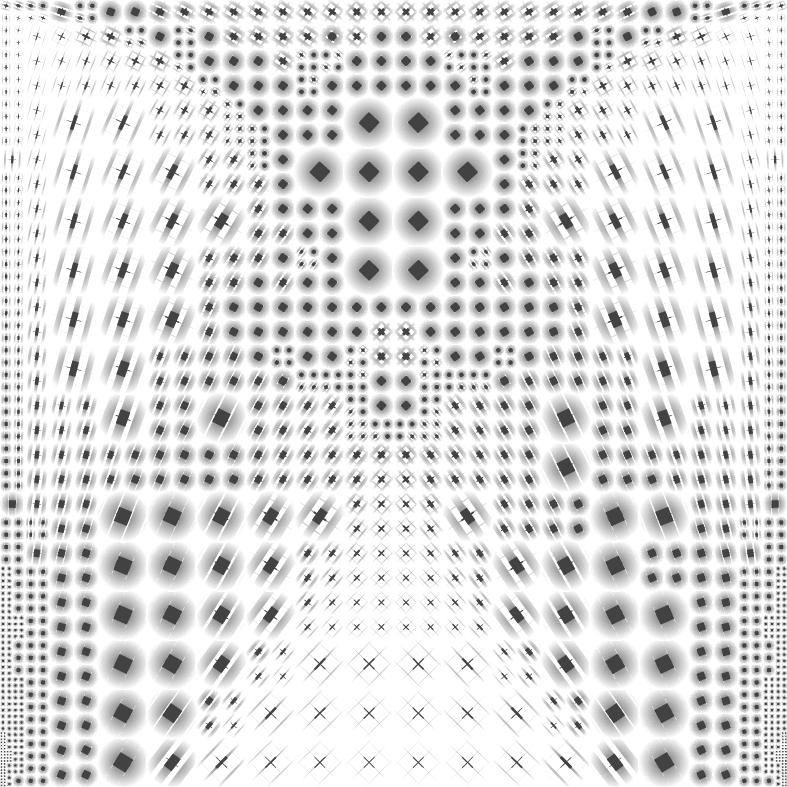}
\hfill
\includegraphics[width=.19\linewidth]{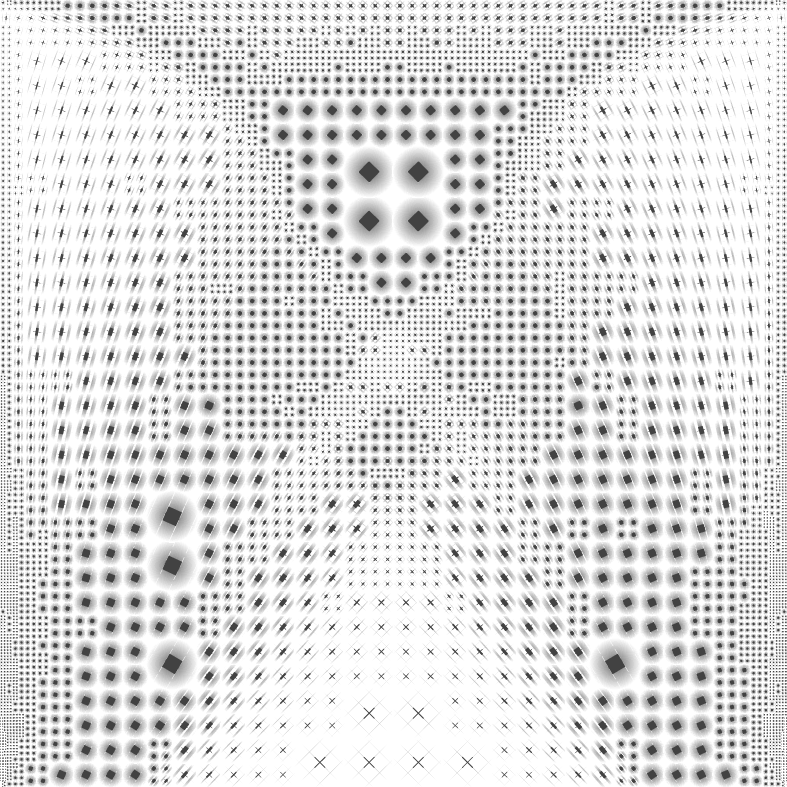}
\hfill
\includegraphics[width=.19\linewidth]{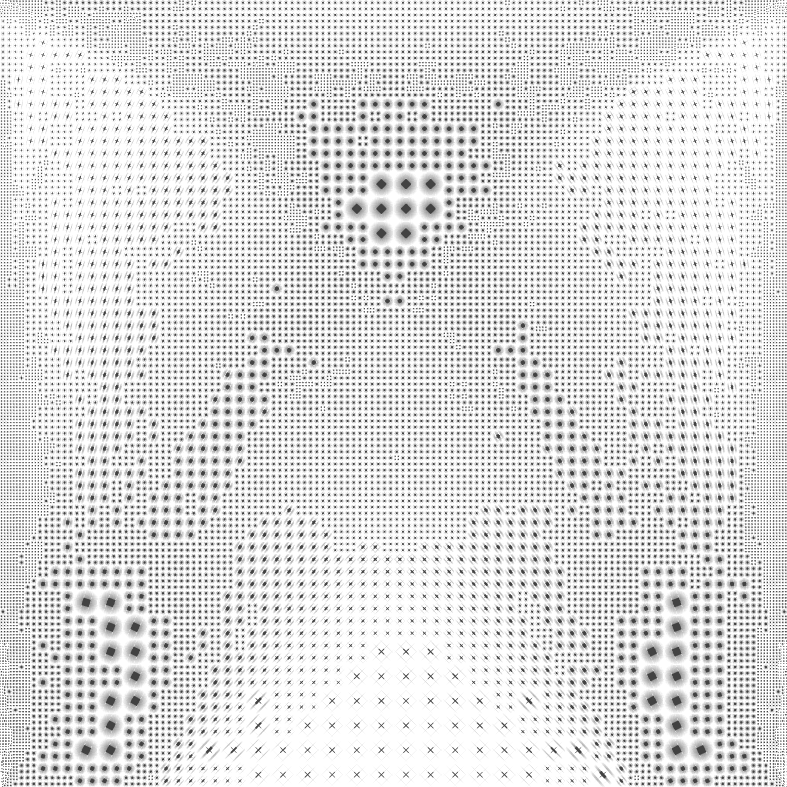} \\ \smallskip

\includegraphics[width=.19\linewidth]{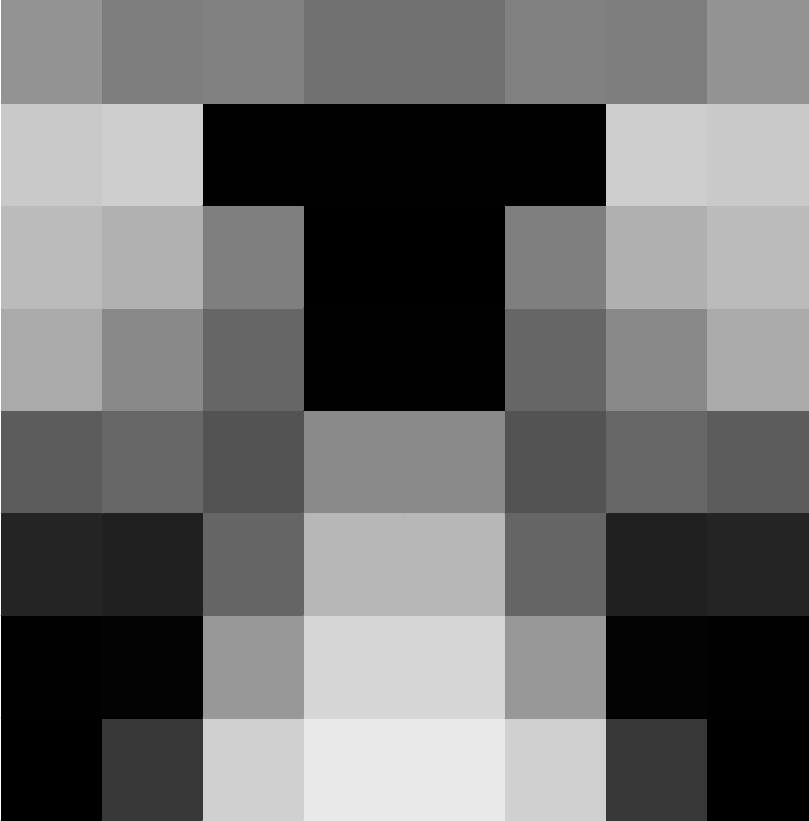}
\hfill
\includegraphics[width=.19\linewidth]{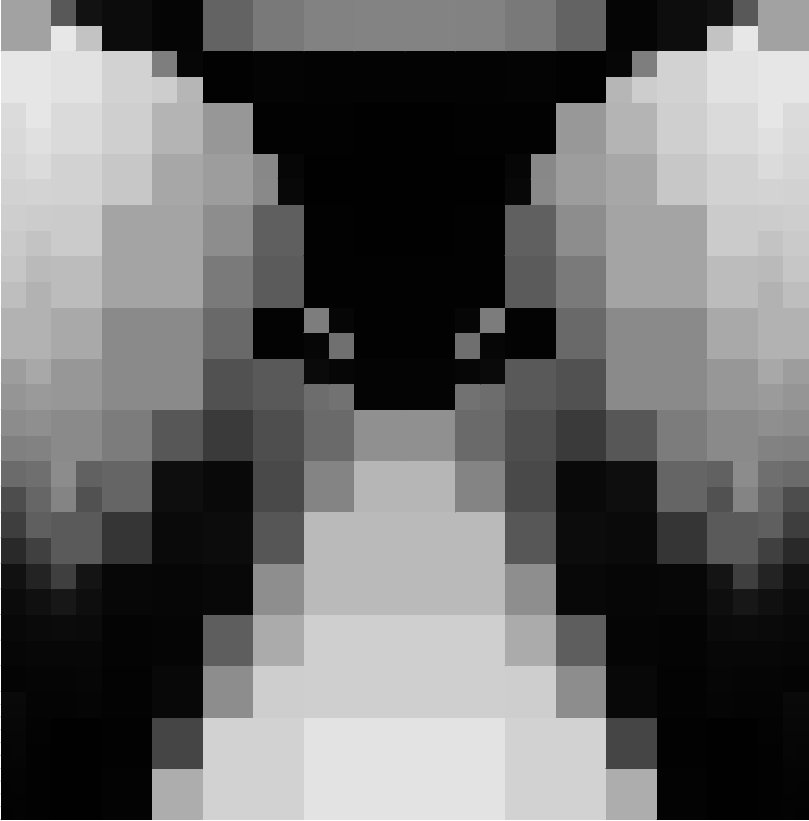}
\hfill
\includegraphics[width=.19\linewidth]{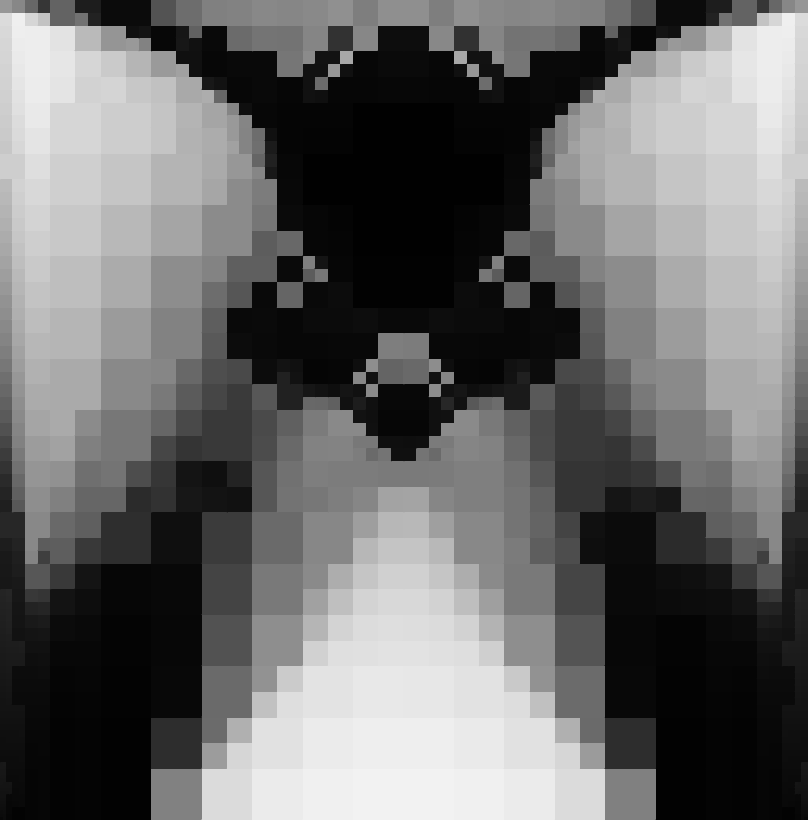}
\hfill
\includegraphics[width=.19\linewidth]{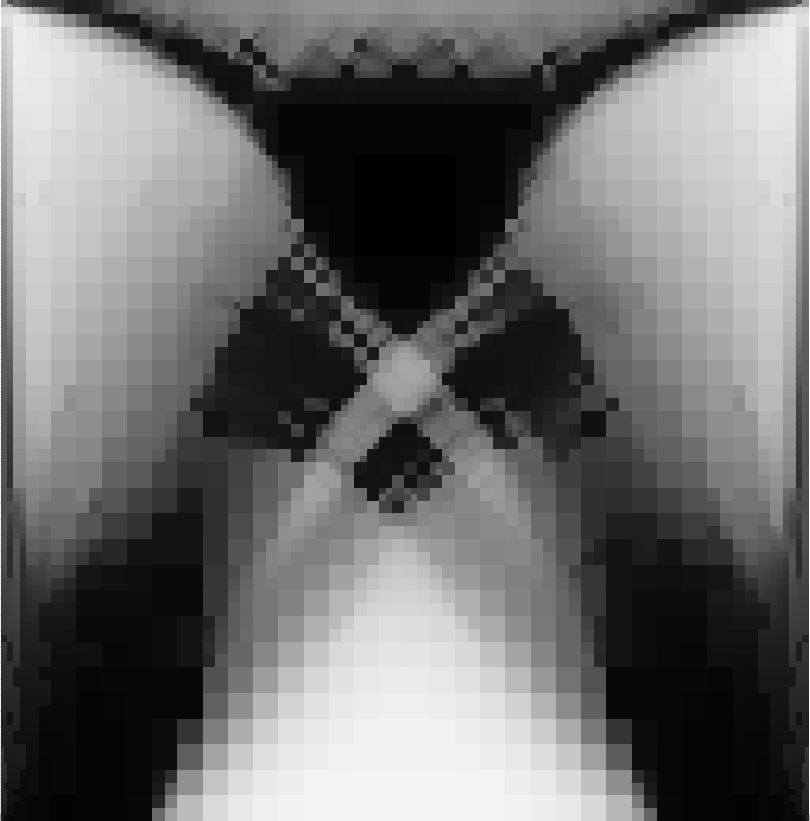}
\hfill
\includegraphics[width=.19\linewidth]{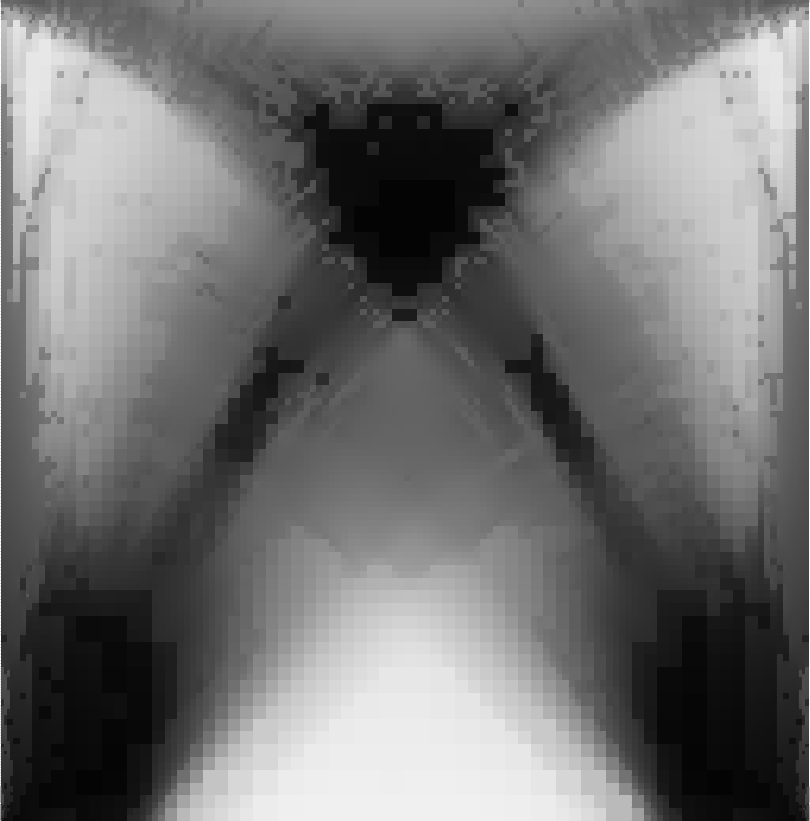} \\ \smallskip

\includegraphics[width=.19\linewidth]{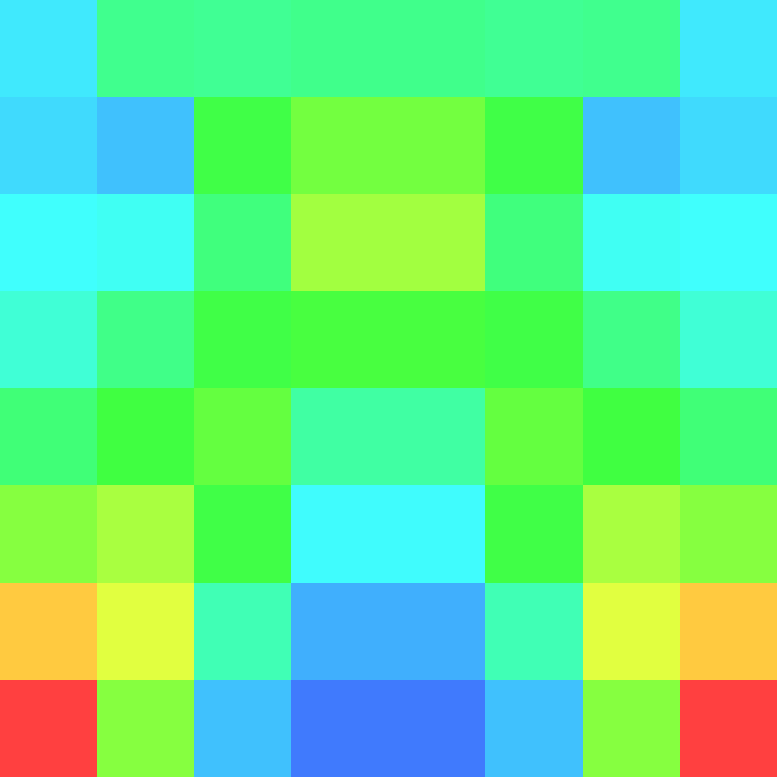}
\hfill
\includegraphics[width=.19\linewidth]{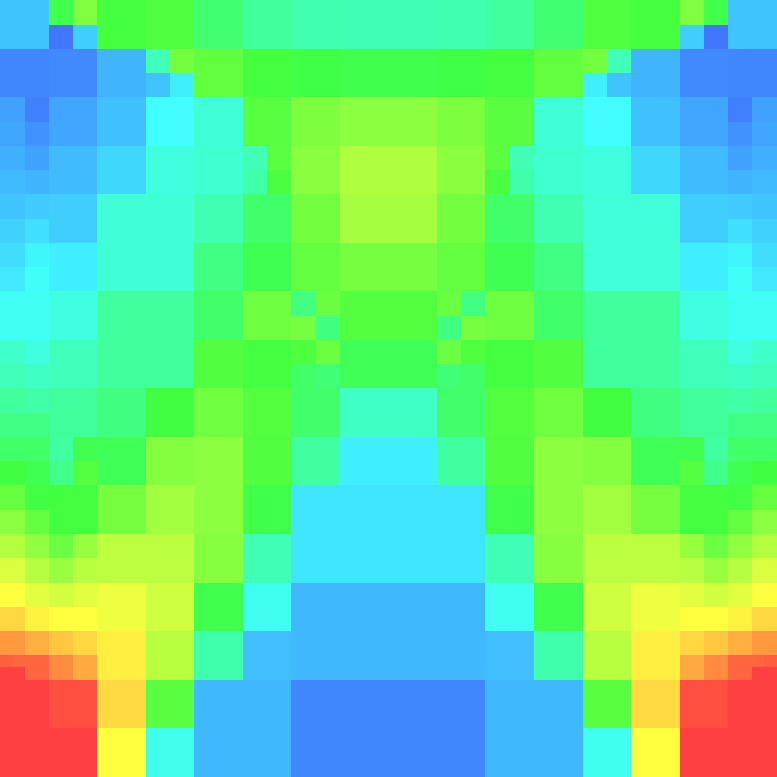}
\hfill
\includegraphics[width=.19\linewidth]{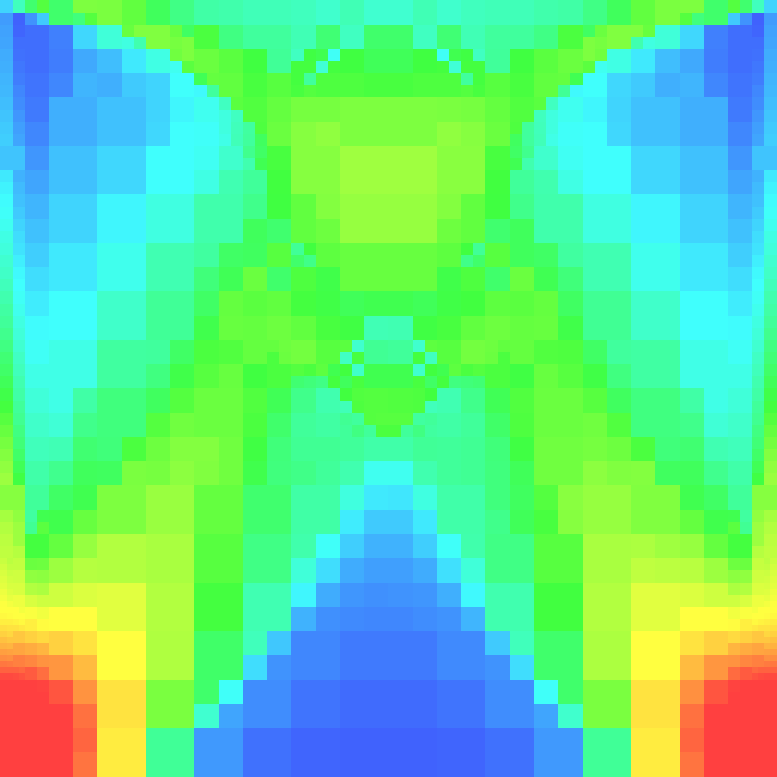}
\hfill
\includegraphics[width=.19\linewidth]{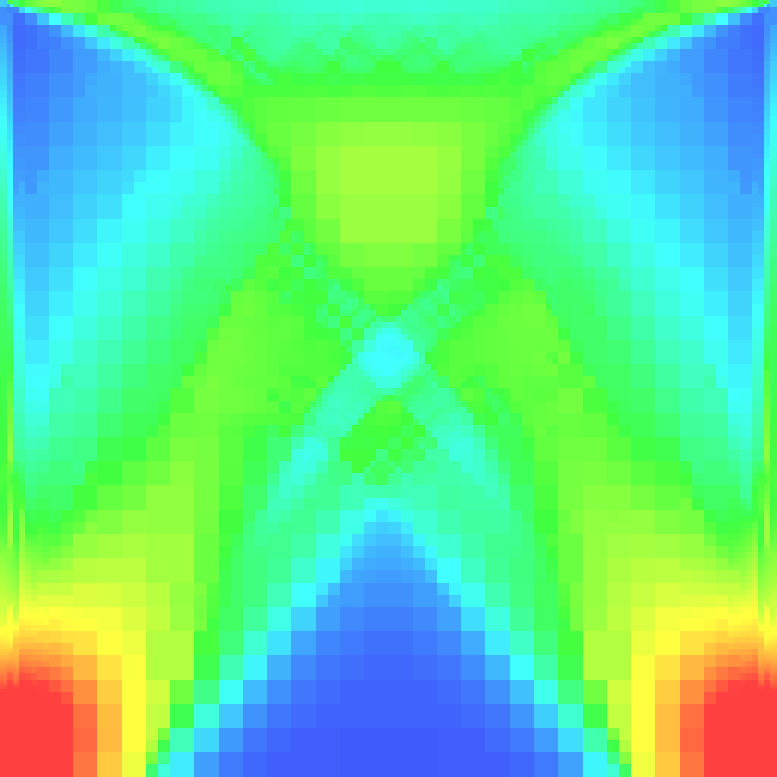}
\hfill
\includegraphics[width=.19\linewidth]{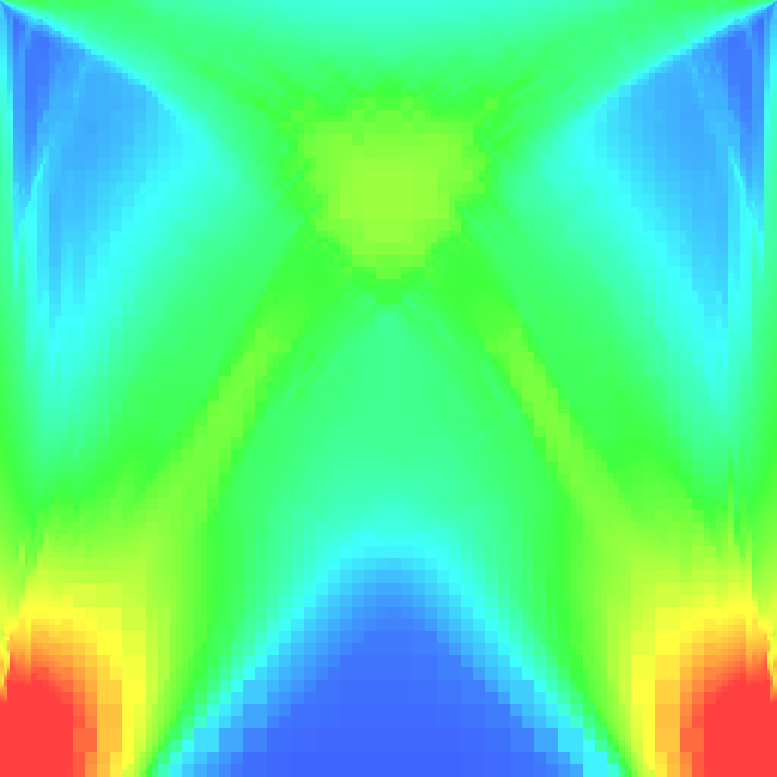} \\ \smallskip

\includegraphics[width=.19\linewidth]{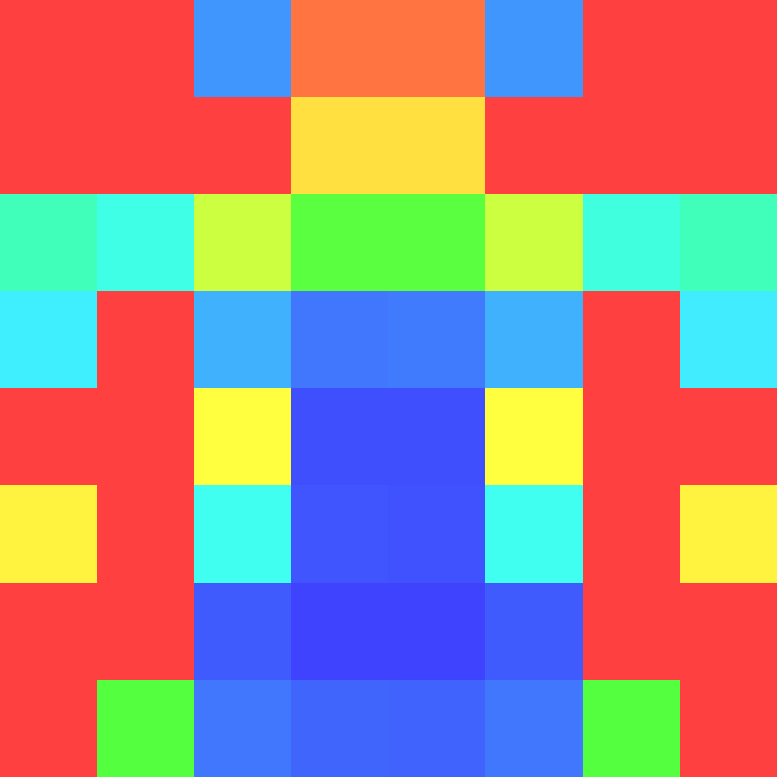}
\hfill
\includegraphics[width=.19\linewidth]{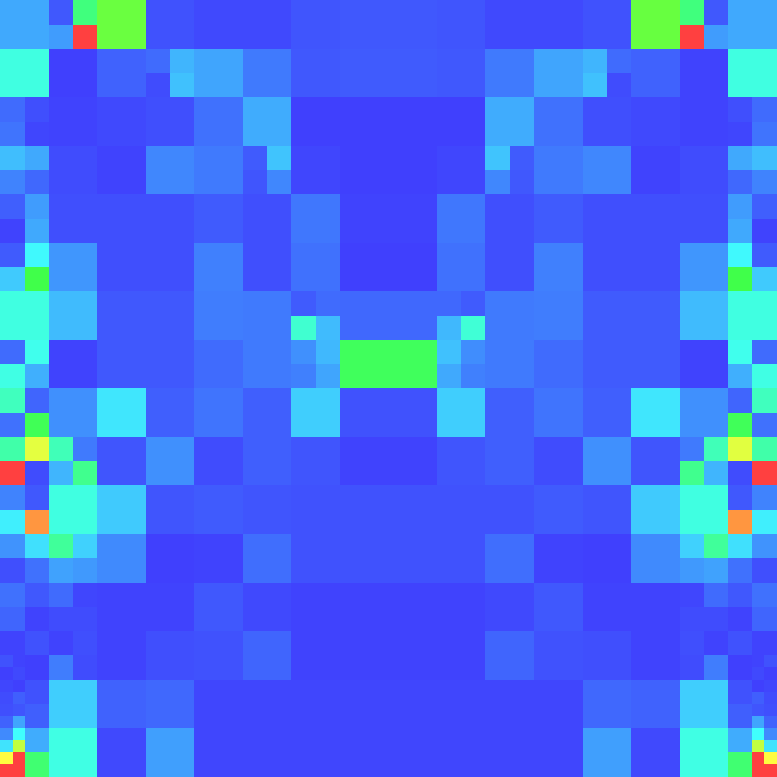}
\hfill
\includegraphics[width=.19\linewidth]{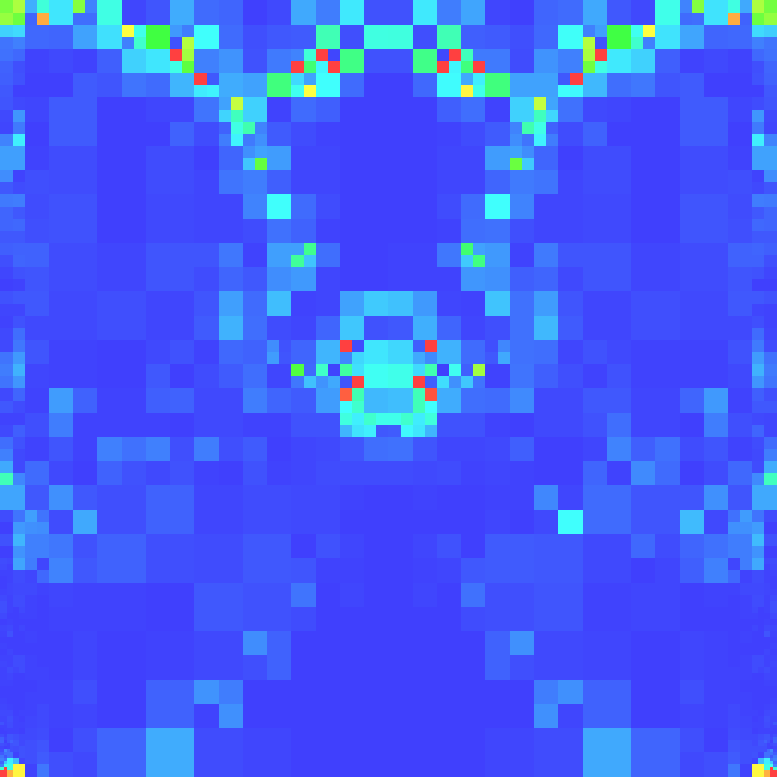}
\hfill
\includegraphics[width=.19\linewidth]{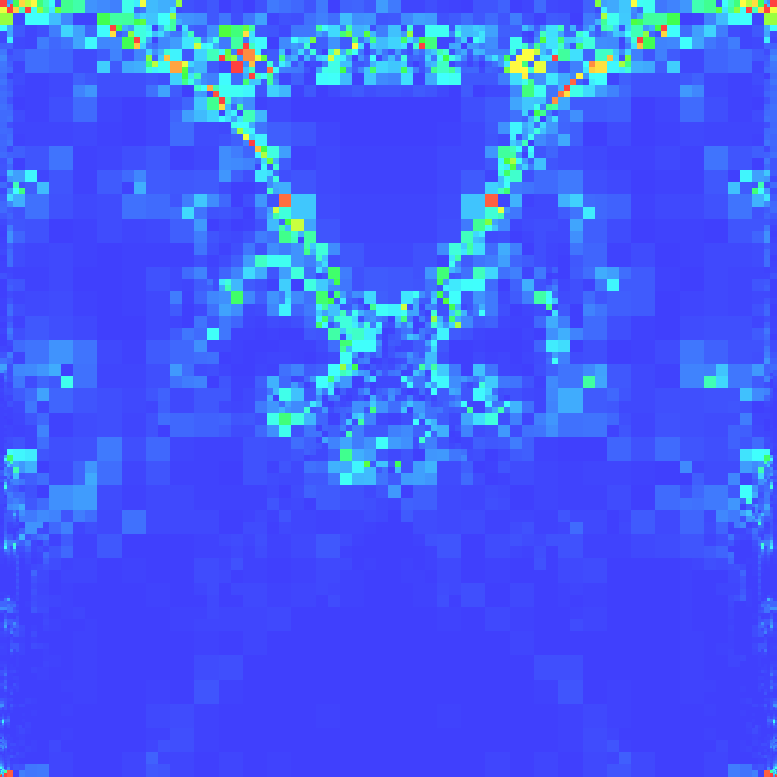}
\hfill
\includegraphics[width=.19\linewidth]{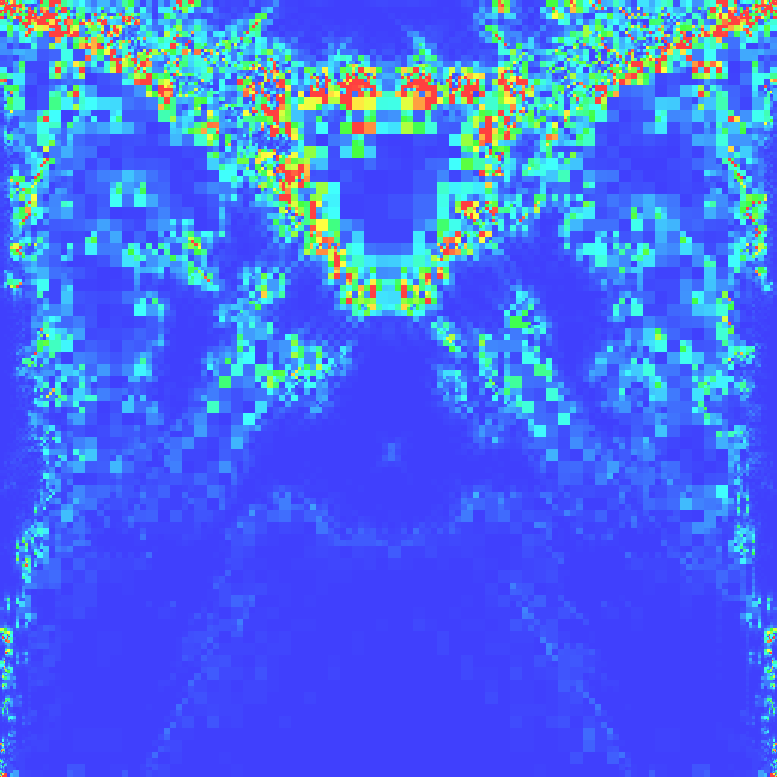} \\

\caption[]{For the carrier plate scenario we depict from top to bottom: the adaptively refined grid  after $0,4,7,10,13$ refinement steps;
a visualization of the optimized microstructure of the two-scale model using an iconic representation with the periodically extended perforations in white and the truss geometry on the fundamental cell in black on a uniformly gray background;
the volume density on the elements of the macroscopic mesh;
the von Mises stress color coded with \raisebox{1pt}{\includegraphics[height=4pt]{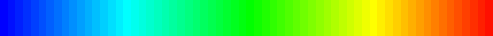}} for values in  $[0,6]$,
and the element wise error indicator $\tilde \eta_E^{\strain} + \tilde \eta_{\partial E}^{\strain}$
with the same color coding now for values in $[0,0.4]$.}
\label{fig:res_carrier}
\end{figure}

\begin{table}[!ht]
\begin{center}
\begin{tabular}{l | l l l | l | l l }
 \#Refs & Edge  $\tilde \eta_{\partial E}^{\strain}$ & Volume $\tilde \eta_E^{\strain}$ & Model $\fr12 \tilde \eta_E^{\etensor}$
        & Total $\tilde \eta_E(\sstrain_h,\stensor_h)$ & $J[\stensor,\sstrain]$ & \#El \\ \hline
 0      &  0.216758 & 0.371898 & 0.726425 & 1.315083 & 2.163037 & 64   \\
 1      &  0.089962 & 0.109742 & 0.496442 & 0.696147 & 2.142808 & 100  \\
 2      &  0.080205 & 0.107282 & 0.410473 & 0.597961 & 2.026496 & 154  \\
 3      &  0.045338 & 0.057250 & 0.355282 & 0.457871 & 2.006264 & 262  \\
 4      &  0.032952 & 0.030399 & 0.290097 & 0.353449 & 1.992369 & 382  \\
 5      &  0.024614 & 0.021315 & 0.216723 & 0.262653 & 1.941411 & 598  \\
 6      &  0.021418 & 0.022156 & 0.188472 & 0.232046 & 1.932662 & 868  \\
 7      &  0.020352 & 0.020339 & 0.165938 & 0.206630 & 1.906596 & 1318 \\
 8      &  0.019183 & 0.018444 & 0.140885 & 0.178513 & 1.901717 & 2077 \\
 9      &  0.020706 & 0.018520 & 0.128235 & 0.167461 & 1.890917 & 3241 \\
10      &  0.022799 & 0.019582 & 0.107348 & 0.149730 & 1.898300 & 5002 \\
11      &  0.026824 & 0.021944 & 0.096681 & 0.145451 & 1.901327 & 7576 \\
12      &  0.031593 & 0.025283 & 0.100130 & 0.157007 & 1.914866 & 11365\\
13      &  0.052313 & 0.041381 & 0.129384 & 0.223079 & 1.933688 & 16903\\
14      &  0.079236 & 0.061311 & 0.151317 & 0.291865 & 1.952025 & 25558
\end{tabular}
\end{center}
\caption{Components of the error indicator for the carrier plate scenario and each refinement step.
The optimal compliance cost based on extrapolation of values obtained from a series of computations for the sequential lamination model on uniform grids is $1.83992$.
At steps 7 to 11 the value for the two-scale model differs from the optimal one by $2$ to $3\%$.}
\label{fig:res_carrier_error}
\end{table}

\begin{table}[!ht]
\begin{center}
\begin{tabular}{l | l l l | l | l l }
 \#Refs & Edge  $\tilde \eta_{\partial E}^{\strain}$ & Volume $\tilde \eta_E^{\strain}$ & Model $\fr12 \tilde \eta_E^{\etensor}$
        & Total $\tilde \eta_E(\sstrain_h,\stensor_h)$ & $J[\stensor,\sstrain]$ & \#El \\ \hline
 0      &  0.216758   & 0.371898 & 0.726425 & 1.315083 & 2.163037 & 64\\
 1      &  0.117318   & 0.131675 & 0.469026 & 0.718020 & 1.992764 & 256\\
 2      &  0.026711   & 0.028785 & 0.374861 & 0.430358 & 1.922647 & 1024\\
 3      &  0.010379   & 0.009895 & 0.207756 & 0.228032 & 1.903045 & 4096 \\
 4      &  0.004076   & 0.004095 & 0.220584 & 0.228756 & 1.945240 & 16384
\end{tabular}
\end{center}
\caption{Components of the error indicator for the carrier plate scenario using uniform refinement.}
\label{fig:res_carrier_erroruni}
\end{table}

\begin{figure}[!ht]
\hspace*{.19\linewidth}
\hfill
\includegraphics[width=.19\linewidth]{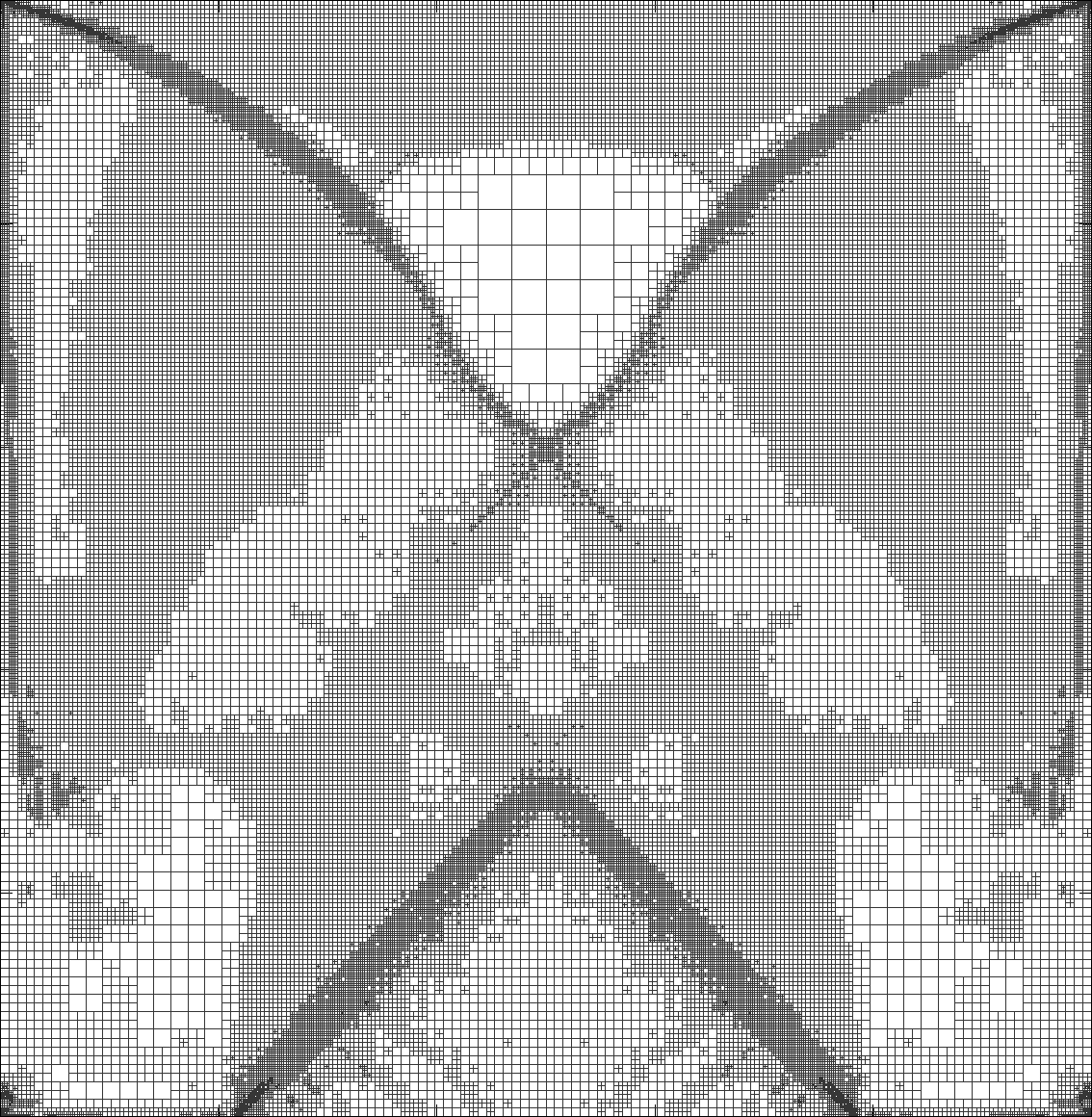}
\hfill
\includegraphics[width=.19\linewidth]{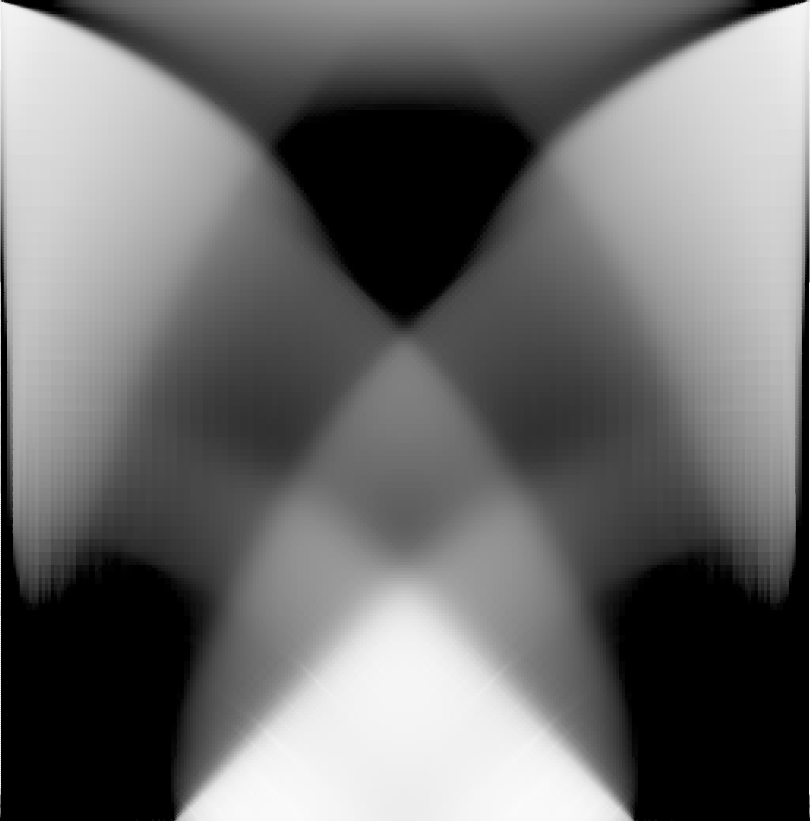}
\hfill
\includegraphics[width=.19\linewidth]{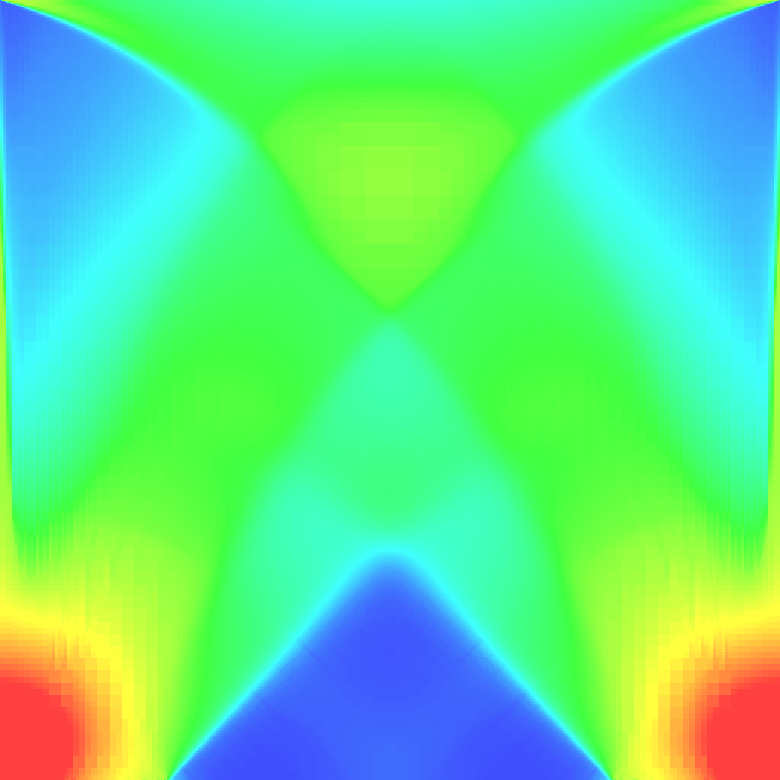}
\hfill
\hspace*{.19\linewidth}
\caption[]{For comparison grid, density, and von Mises stress with color coding as in Figure~\ref{fig:res_carrier} for values in $[0,6]$
for the sequential lamination model after 24 refinement steps, based on the approach described in \cite{GeRu15}.}
\label{fig:res_carrier_laminates}
\end{figure}

%
%

\begin{figure}[!ht]
\includegraphics[width=.32\linewidth]{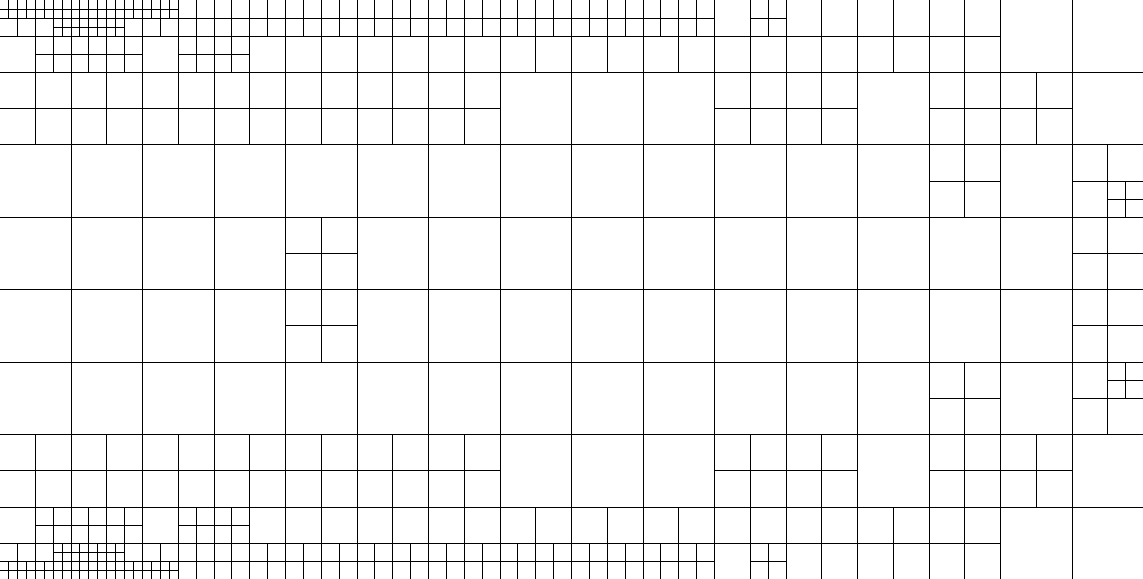}
\hfill
\includegraphics[width=.32\linewidth]{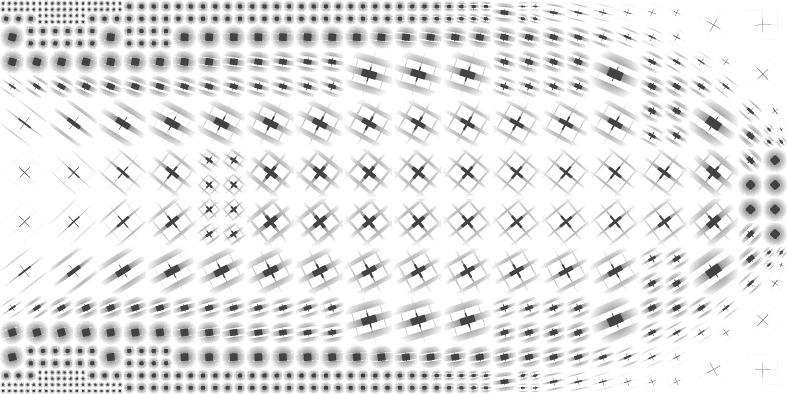}
\hfill
\includegraphics[width=.32\linewidth]{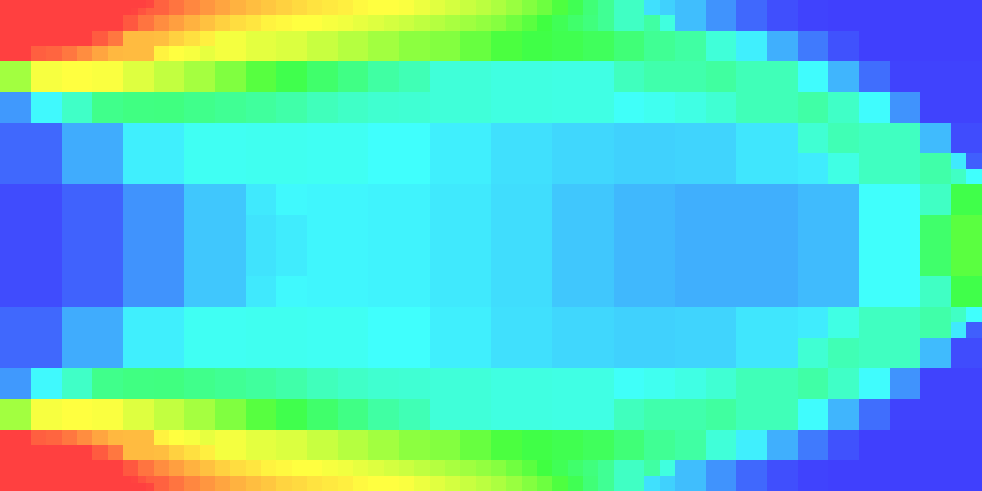} \\ \smallskip

\includegraphics[width=.32\linewidth]{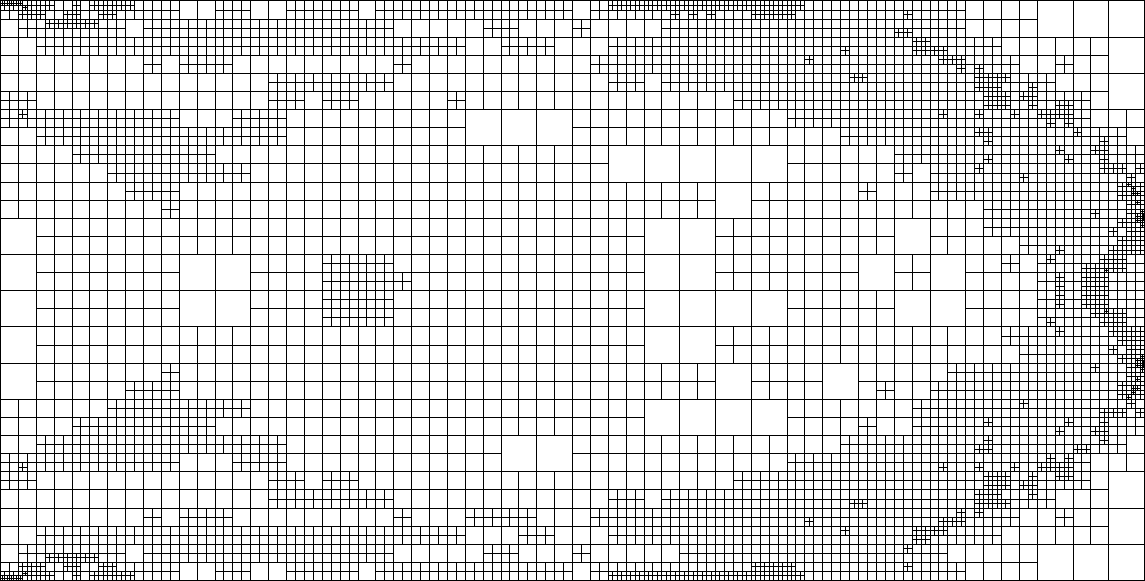}
\hfill
\includegraphics[width=.32\linewidth]{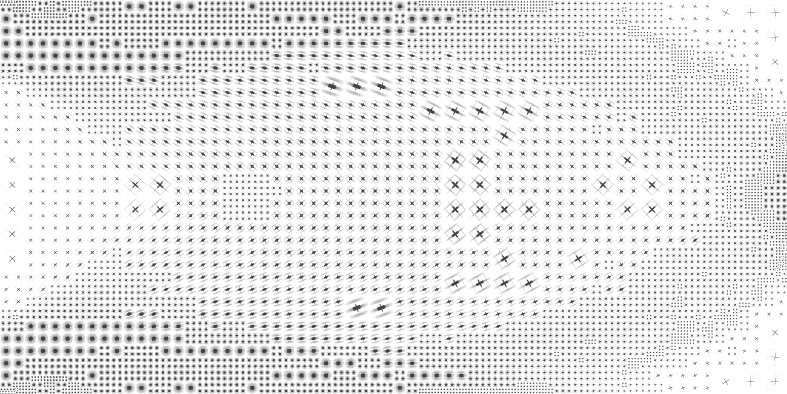}
\hfill
\includegraphics[width=.32\linewidth]{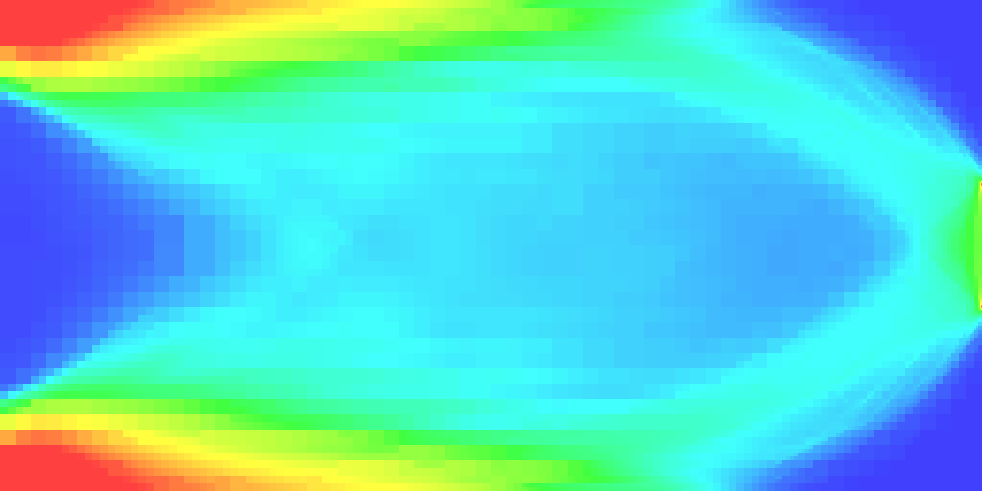}
\caption[]{For the cantilever scenario after 4 / 9 refinement steps: adaptively refined grid; visualization of the two-scale model as in Figure~\ref{fig:res_carrier};
von Mises stress with color coding as in Figure~\ref{fig:res_carrier} for values in $[0,3]$.}
\label{fig:res_cantilever}
\end{figure}

\begin{table}[!ht]
\begin{center}
\begin{tabular}{r | c@{\,}c@{\,}c | c@{\,}c@{\,}c | c@{\,}c@{\,}c }
        & \multicolumn{3}{|c}{\textbf{Cantilever}} & \multicolumn{3}{|c}{\textbf{Bridge}} & \multicolumn{3}{|c}{\textbf{L-Shape}} \\ \hline
 \#Refs &$\tilde \eta_E^{\strain} + \tilde \eta_{\partial E}^{\strain}$
        &$\fr12 \tilde \eta_E^{\etensor}$ & $J[\stensor,\sstrain]$
        &$\tilde \eta_E^{\strain} + \tilde \eta_{\partial E}^{\strain}$
        &$\fr12 \tilde \eta_E^{\etensor}$ & $J[\stensor,\sstrain]$
        &$\tilde \eta_E^{\strain} + \tilde \eta_{\partial E}^{\strain}$
        &$\fr12 \tilde \eta_E^{\etensor}$ & $J[\stensor,\sstrain]$ \\ \hline
 0      & 0.0243 & 0.0706 & 0.2416 &  0.7390 & 0.3552 & 0.9239 & 0.0668 & 0.0729 & 0.2848 \\
 1      & 0.0114 & 0.0503 & 0.2083 &  0.3058 & 0.3830 & 0.8466 & 0.0332 & 0.0508 & 0.2459 \\
 2      & 0.0068 & 0.0422 & 0.1924 &  0.2132 & 0.2909 & 0.7931 & 0.0206 & 0.0359 & 0.2270 \\
 3      & 0.0044 & 0.0324 & 0.1913 &  0.1011 & 0.2421 & 0.7632 & 0.0126 & 0.0322 & 0.2213 \\
 4      & 0.0049 & 0.0303 & 0.1889 &  0.0519 & 0.2139 & 0.7436 & 0.0088 & 0.0264 & 0.2198 \\
 5      & 0.0053 & 0.0250 & 0.1807 &  0.0413 & 0.2209 & 0.7355 & 0.0053 & 0.0239 & 0.2190 \\
 6      & 0.0050 & 0.0220 & 0.1772 &  0.0436 & 0.2620 & 0.7354 & 0.0088 & 0.0213 & 0.2147 \\
 7      & 0.0067 & 0.0182 & 0.1761 &  0.3006 & 0.4511 & 0.8585 & 0.0108 & 0.0221 & 0.2113 \\
 8      & 0.0111 & 0.0218 & 0.1776 &  0.0804 & 0.3054 & 0.7570 & 0.0185 & 0.0291 & 0.2123 \\
 9      & 0.0164 & 0.0242 & 0.1794 &  0.1393 & 0.3207 & 0.7785 & 0.0222 & 0.0310 & 0.2121
\end{tabular}
\end{center}
\caption{Components of the error indicator for each refinement step for the cantilever, the bridge, and the L-shaped domain.
}
\label{fig:res_other_error}
\end{table}

%
%

\begin{figure}[!ht]
\includegraphics[width=.32\linewidth]{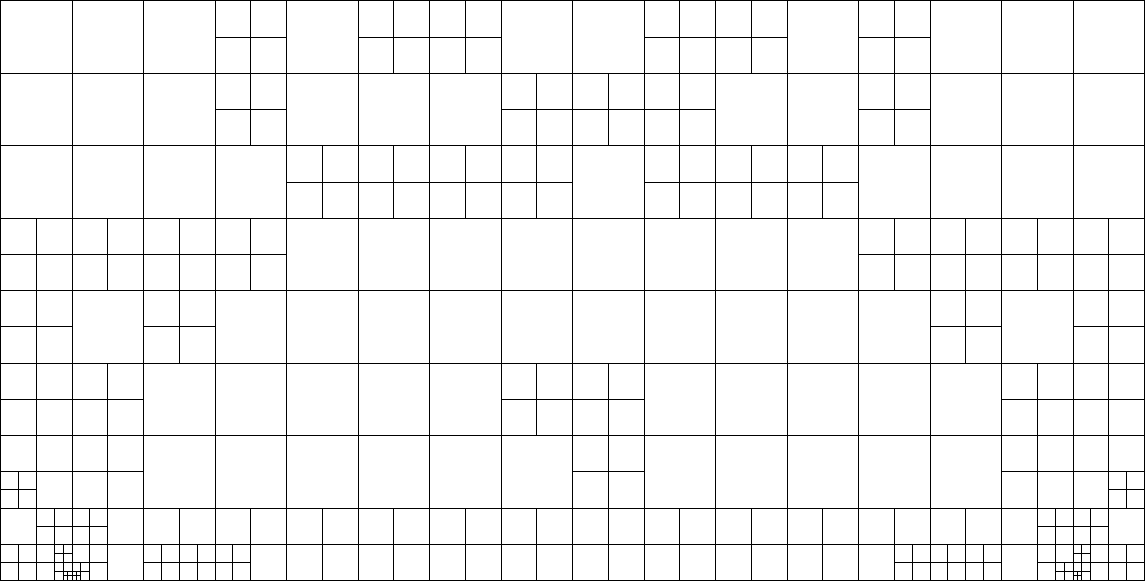}
\hfill
\includegraphics[width=.32\linewidth]{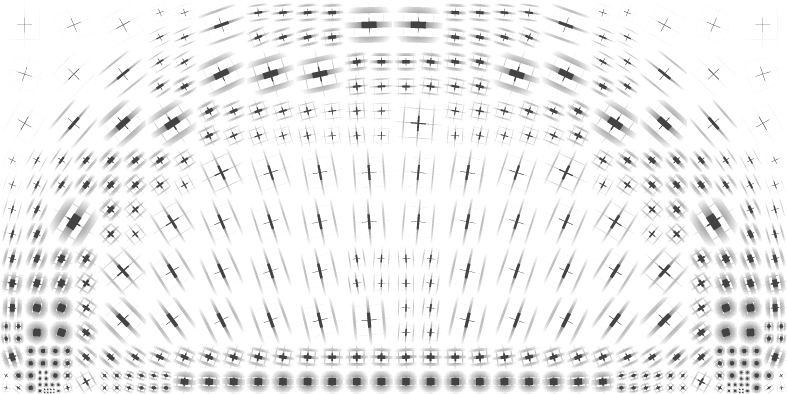}
\hfill
\includegraphics[width=.32\linewidth]{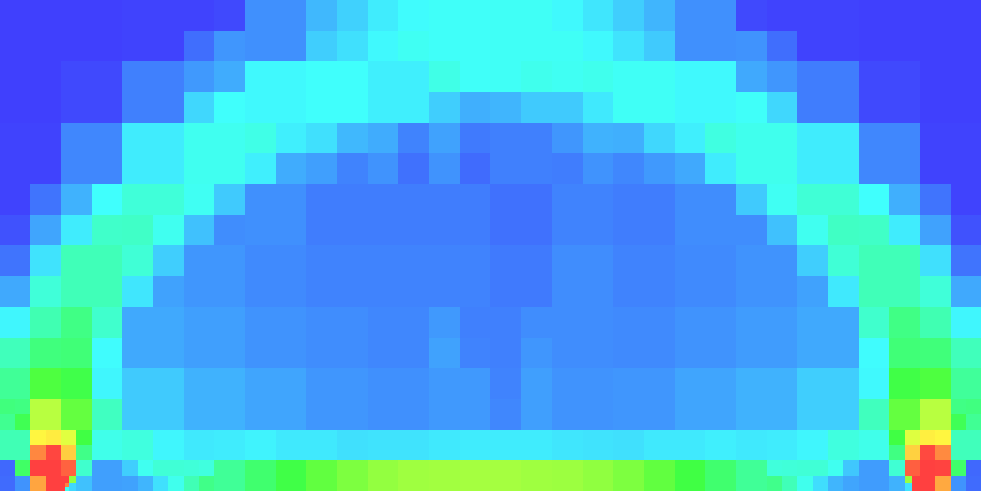} \\ \smallskip

\includegraphics[width=.32\linewidth]{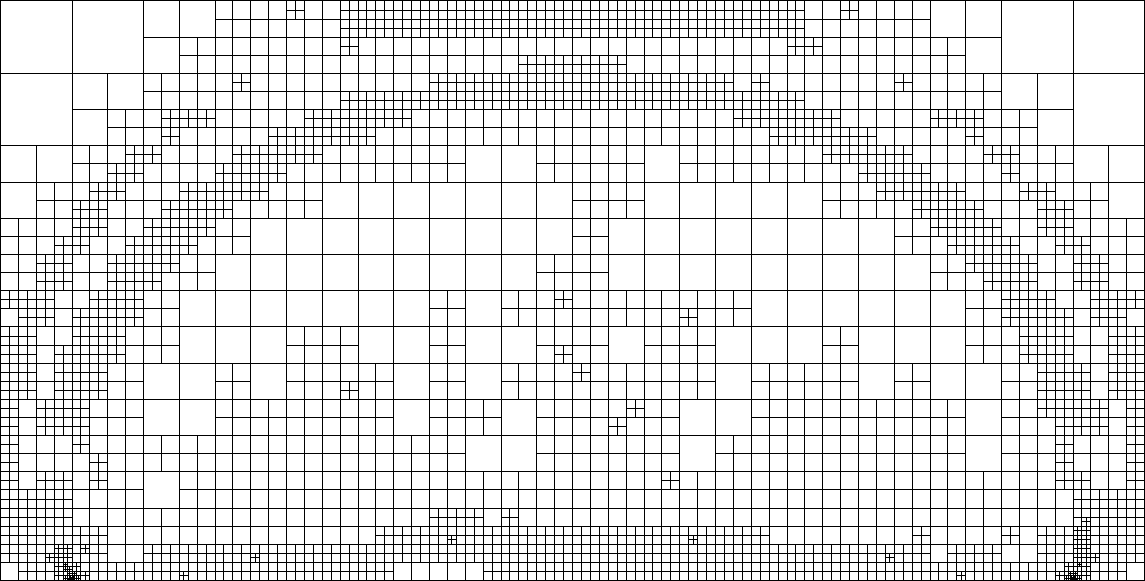}
\hfill
\includegraphics[width=.32\linewidth]{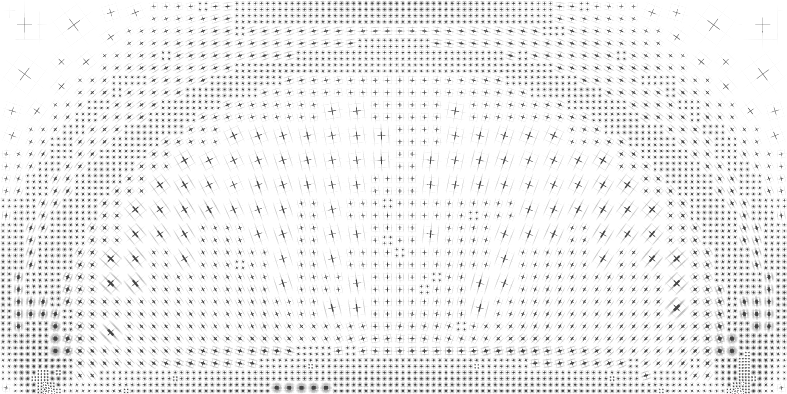}
\hfill
\includegraphics[width=.32\linewidth]{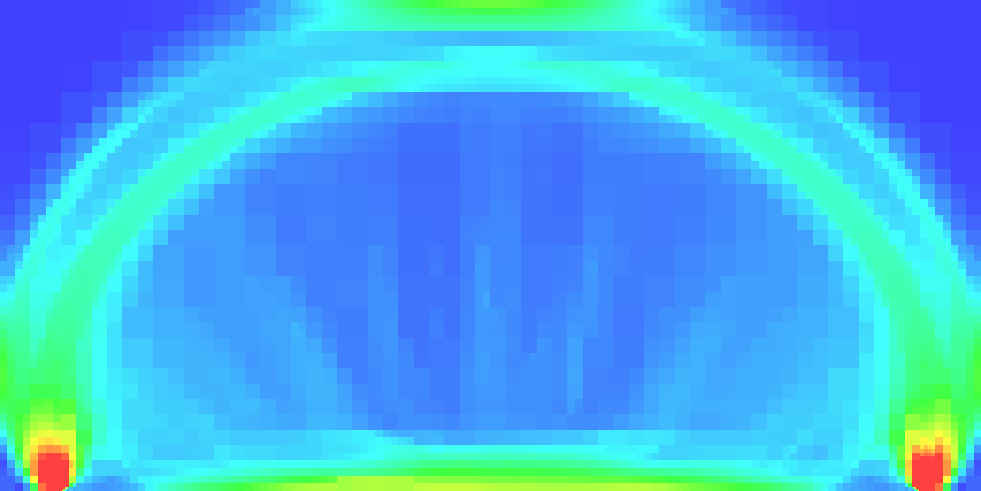}
\caption[]{For the bridge scenario after 4 / 9 refinement steps: adaptively refined grid; visualization of the two-scale model as in Figure~\ref{fig:res_carrier};
von Mises stress with color coding as in Figure~\ref{fig:res_carrier} for values in $[0,8]$.
}
\label{fig:res_bridge}
\end{figure}

%
%

\begin{figure}[!ht]
\includegraphics[width=.32\linewidth]{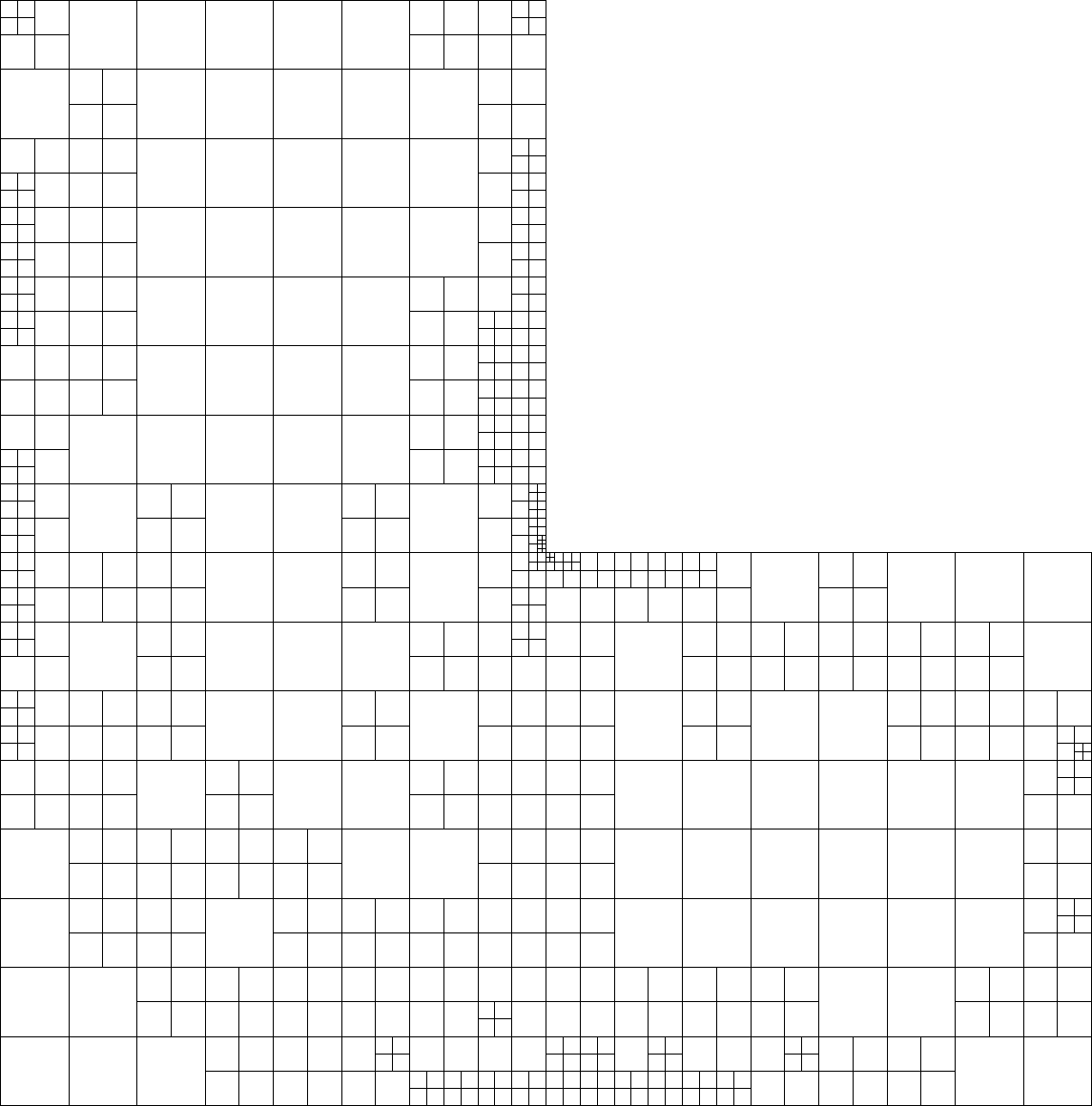}
\hfill
\includegraphics[width=.32\linewidth]{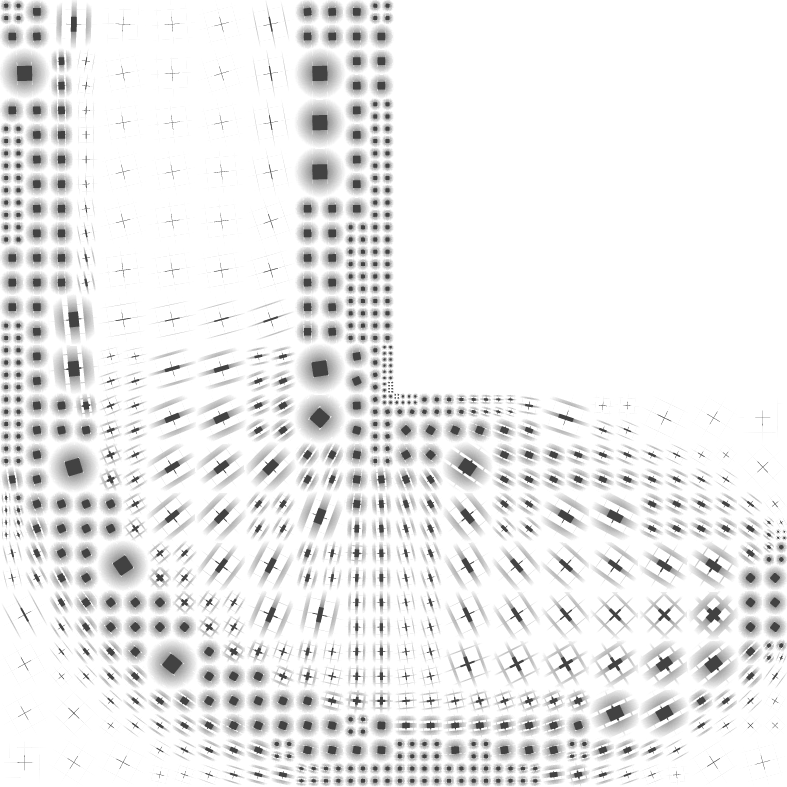}
\hfill
\includegraphics[width=.32\linewidth]{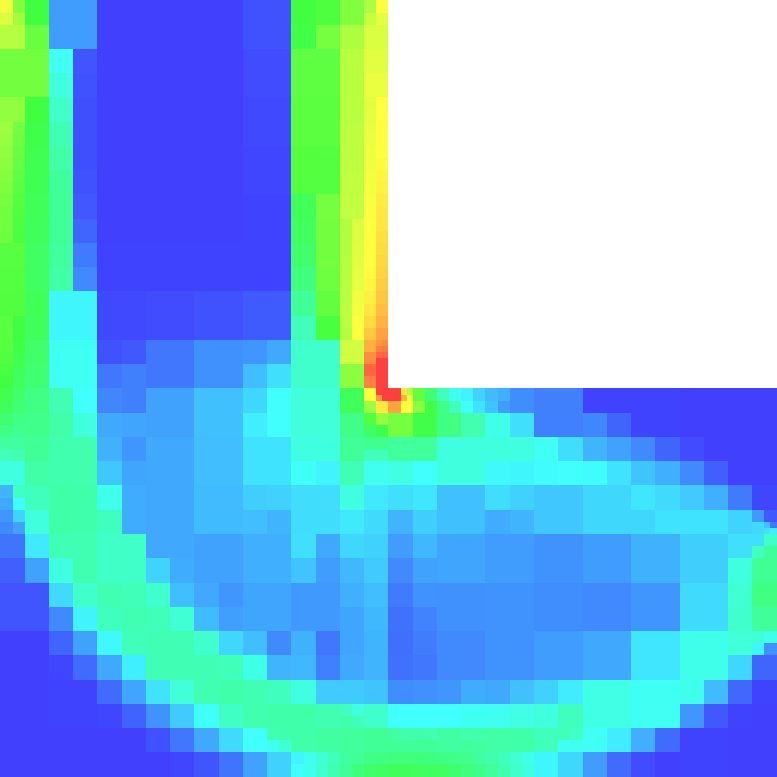} \\ \smallskip

\includegraphics[width=.32\linewidth]{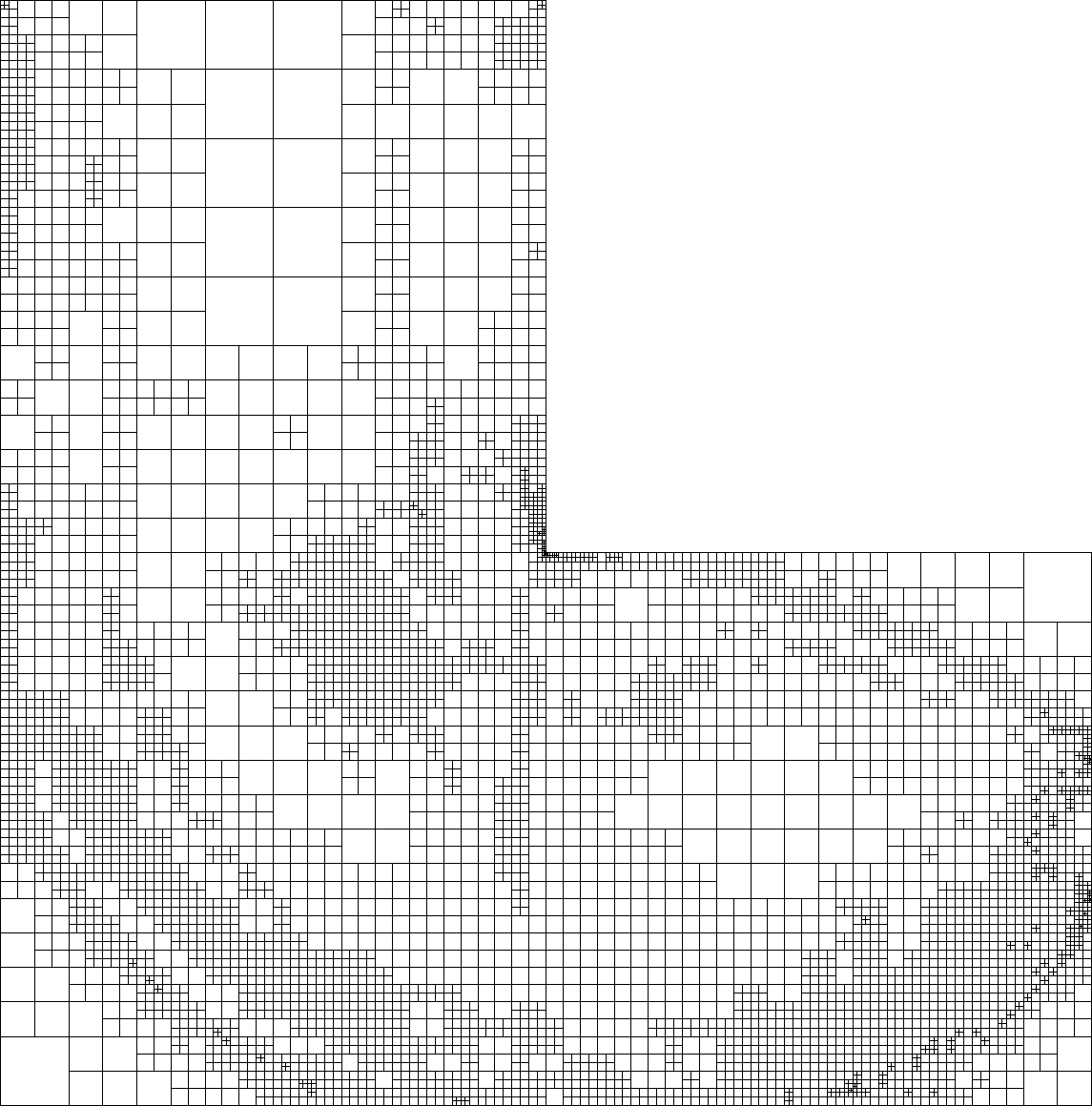}
\hfill
\includegraphics[width=.32\linewidth]{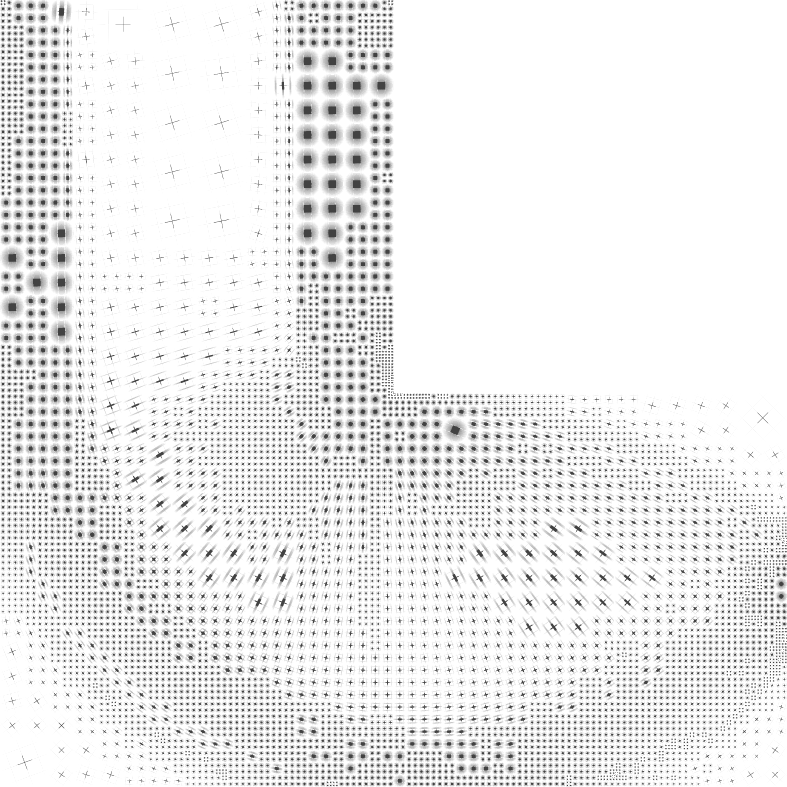}
\hfill
\includegraphics[width=.32\linewidth]{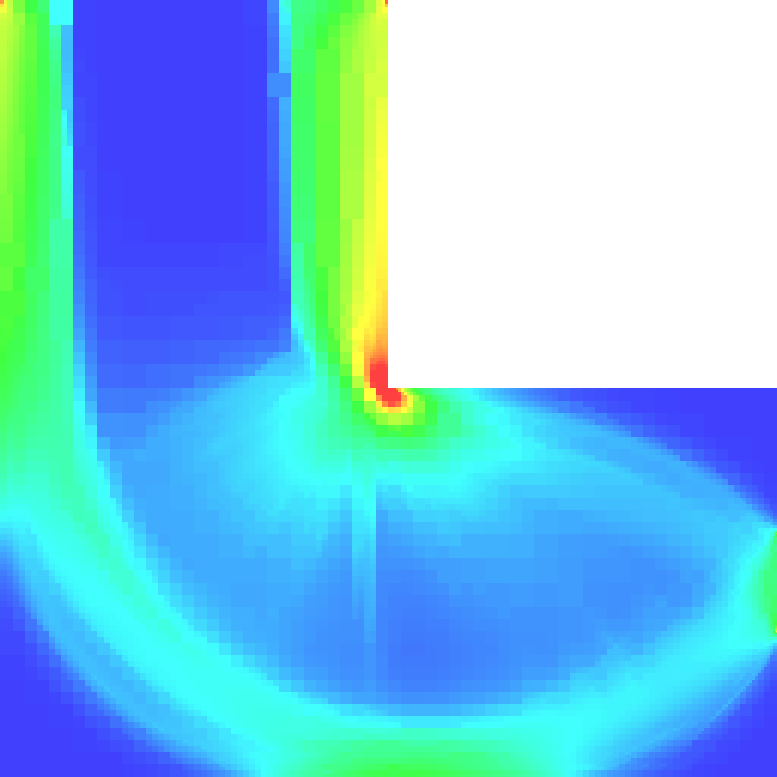}
\caption[]{For the L-shaped domain after 4 / 9 refinement steps: adaptively refined grid; visualization of the two-scale model as in Figure~\ref{fig:res_carrier};
von Mises stress with color coding as in Figure~\ref{fig:res_carrier} for values in $[0,4]$.
}
\label{fig:res_ldomain}
\end{figure}

\section{Conclusion} \label{sec:sum}
We presented a two-scale approach to shape optimization based on simple
and physically realizable microstructures, consisting of two orthogonal trusses
whose width and orientation are optimized. We gave a posteriori bounds of the resulting
discretization and modeling error. We presented a numerical implementation of the scheme,
based on the boundary element method at the microscale and finite elements at the macroscale,
which uses the estimated error to guide grid adaptivity, and employed
it to study several standard model problems in shape optimization.

Our results illustrate good convergence of the objective functional and of the
general structure of the shape. However, after an initial decrease, the error
indicator increases when the grid becomes finer and finer. This is attributed
to the emergence of oscillations in the microscopic parameters at the grid scale
and should be considered as an indication of the limits of the approximate microscopic model
used in the algorithm. The proposed error indicators robustly reports this turning point where the algorithm could be stopped.
Hence, the proposed algorithm allows to identify an adaptive mesh on which the physical realization of a close to optimal
microstructure is mechanically easier and possibly more realistic, and helps to avoid overrefinement.

Furthermore, let us emphasize that locally on each cell of the macroscopic finite element mesh the periodic truss microstructure is  mechanically constructable.
Blending between the microstructures of adjacent cells will have an impact on the cost of the overall construction, as discussed by Allaire and Kohn \cite{AlKo93}
and Bends{\o}e \etal\  \cite{BeDi93}.
In that respect, it is advisable to start from the two scale formulation and use it to generate a good initial configuration for a single scale optimization,
which will then result in a geometry which can actually be manufactured. The presented weighted residual approach for a posteriori error estimates would then translate also to this single scale
model including the balance between model error reduction and potential increase of the discretization error discussed above.

\section*{Acknowledgments}
This work has been supported by the \emph{Deutsche Forschungsgemeinschaft} through
\emph{Collaborative Research Centre 1060 -- The Mathematics of Emergent Effects}.

\bibliographystyle{plain-noname}
\bibliography{bibtex/all,bibtex/own,bibtex/library,local,neuerefs}

\def\polhk#1{\setbox0=\hbox{#1}{\ooalign{\hidewidth
  \lower1.5ex\hbox{`}\hidewidth\crcr\unhbox0}}}
\begin{thebibliography}{10}

\bibitem{AlBoFr97}
G.~Allaire, E.~Bonnetier, G.~Francfort, and F.~Jouve.
\newblock Shape optimization by the homogenization method.
\newblock {\em Numerische Mathematik}, 76:27--68, 1997.

\bibitem{AlKo93}
G.~Allaire and R.~V. Kohn.
\newblock Optimal design for minimum weight and compliance in plane stress
  using extremal microstructures.
\newblock {\em European J. Mech. A Solids}, 12(6):839--878, 1993.

\bibitem{Al02}
G.~Allaire.
\newblock {\em Shape optimization by the homogenization method}, volume 146 of
  {\em Applied Mathematical Sciences}.
\newblock Springer-Verlag, New York, 2002.

\bibitem{AlKo93a}
G.~Allaire and R.~V. Kohn.
\newblock Explicit optimal bounds on the elastic energy of a two-phase
  composite in two space dimensions.
\newblock {\em Quart. Appl. Math.}, 51(4):675--699, 1993.

\bibitem{AtCoGe12}
P.~Atwal, S.~Conti, B.~Geihe, M.~Pach, M.~Rumpf, and R.~Schultz.
\newblock On shape optimization with stochastic loadings.
\newblock In G.~Leugering, S.~Engell, A.~Griewank, M.~Hinze, R.~Rannacher,
  V.~Schulz, M.~Ulbrich, and S.~Ulbrich, editors, {\em Constrained Optimization
  and Optimal Control for Partial Differential Equations}, volume 160 of {\em
  International Series of Numerical Mathematics}, chapter~2, pages 215--243.
  Springer, Basel, 2012.

\bibitem{Av87}
M.~Avellaneda.
\newblock Optimal bounds and microgeometries for elastic two-phase composites.
\newblock {\em SIAM J. Appl. Math.}, 47(6):1216--1228, 1987.

\bibitem{BaTo10}
C.~Barbarosie and A.-M. Toader.
\newblock Shape and topology optimization for periodic problems. {I}. {T}he
  shape and the topological derivative.
\newblock {\em Struct. Multidiscip. Optim.}, 40(1-6):381--391, 2010.

\bibitem{BaTo10a}
C.~Barbarosie and A.-M. Toader.
\newblock Shape and topology optimization for periodic problems. {II}.
  {O}ptimization algorithm and numerical examples.
\newblock {\em Struct. Multidiscip. Optim.}, 40(1-6):393--408, 2010.

\bibitem{BaTo12}
C.~Barbarosie and A.-M. Toader.
\newblock Optimization of bodies with locally periodic microstructure.
\newblock {\em Mechanics of Advanced Materials and Structures}, 19(4):290--301,
  2012.

\bibitem{BeEsTr11}
R.~Becker, E.~Estecahandy, and D.~Trujillo.
\newblock Weighted marking for goal-oriented adaptive finite element methods.
\newblock {\em SIAM Journal on Numerical Analysis}, 49(6):2451--2469, 2011.

\bibitem{BeKaRa00}
R.~Becker, H.~Kapp, and R.~Rannacher.
\newblock Adaptive finite element methods for optimal control of partial
  differential equations: Basic concept.
\newblock {\em SIAM J. Control Optim.}, 39(1):113--132, 2000.

\bibitem{BeRa97}
R.~Becker and R.~Rannacher.
\newblock A feed-back approach to error control in finite element methods:
  Basic analysis and examples.
\newblock {\em Computational Mechanics}, 5:434--446, 1997.

\bibitem{BeGu94}
M.~P. Bends{\o}e, J.~M. Guedes, R.~B. Haber, P.~Pedersen, and J.~E. Taylor.
\newblock An analytical model to predict optimal material properties in the
  context of optimal structural design.
\newblock {\em Trans. ASME J. Appl. Mech.}, 61(4):930--937, 1994.

\bibitem{Be95a}
M.~P. Bends{\o}e.
\newblock {\em Optimization of structural topology, shape, and material}.
\newblock Springer-Verlag, Berlin, 1995.

\bibitem{BeDi93}
M.~P. Bends{\o}e, A.~D{\'{\i}}az, and N.~Kikuchi.
\newblock Topology and generalized layout optimization of elastic structures.
\newblock In {\em Topology design of structures ({S}esimbra, 1992)}, volume 227
  of {\em NATO Adv. Sci. Inst. Ser. E Appl. Sci.}, pages 159--205. Kluwer Acad.
  Publ., Dordrecht, 1993.

\bibitem{BeKi88}
M.~P. Bends{\o}e and N.~Kikuchi.
\newblock Generating optimal topologies in structural design using a
  homogenization method.
\newblock {\em Computer Methods in Applied Mechanics and Engineering},
  71(2):197--224, 1988.

\bibitem{BeVe09}
O.~Benedix and B.~Vexler.
\newblock A posteriori error estimation and adaptivity for elliptic optimal
  control problems with state constraints.
\newblock {\em Computational Optimization and Applications}, 44:3--25, 2009.

\bibitem{BrDe98}
A.~Braides and A.~Defranceschi.
\newblock {\em Homogenization of Multiple Integrals}.
\newblock Claredon Press, Oxford, 1998.

\bibitem{Braides1998}
A.~Braides and A.~Defranceschi.
\newblock {\em Homogenization of multiple integrals}, volume~12 of {\em Oxford
  Lecture Series in Mathematics and its Applications}.
\newblock The Clarendon Press Oxford University Press, New York, 1998.

\bibitem{BrLiUl12}
C.~Brandenburg, F.~Lindemann, M.~Ulbrich, and S.~Ulbrich.
\newblock Advanced numerical methods for pde constrained optimization with
  application to optimal design in navier stokes flow.
\newblock In G.~Leugering, S.~Engell, A.~Griewank, M.~Hinze, R.~Rannacher,
  V.~Schulz, M.~Ulbrich, and S.~Ulbrich, editors, {\em Constrained Optimization
  and Optimal Control for Partial Differential Equations}, volume 160 of {\em
  International Series of Numerical Mathematics}, pages 257--275. Springer
  Basel, 2012.

\bibitem{BuDa91}
G.~Buttazzo and G.~{Dal Maso}.
\newblock Shape optimization for {D}irichlet problems: Relaxed formulation and
  optimality conditions.
\newblock {\em Applied Mathematics and Optimization}, 23:17--49, 1991.

\bibitem{Cherkaev2000}
A.~Cherkaev.
\newblock {\em Variational methods for structural optimization}, volume 140 of
  {\em Applied Mathematical Sciences}.
\newblock Springer-Verlag, New York, 2000.

\bibitem{Ciarlet1978}
P.~G. Ciarlet.
\newblock {\em The finite element method for elliptic problems}.
\newblock North-Holland Publishing Company, 1978.

\bibitem{DoCi99}
D.~Cioranescu and P.~Donato.
\newblock {\em An Introduction to Homogenization}.
\newblock Oxford University Press, Oxford, 1999.

\bibitem{CoGeRu14}
S.~Conti, B.~Geihe, M.~Rumpf, and R.~Schultz.
\newblock Two-stage stochastic optimization meets two-scale simulation.
\newblock In G.~Leugering, P.~Benner, S.~Engell, A.~Griewank, H.~Harbrecht,
  M.~Hinze, R.~Rannacher, and S.~Ulbrich, editors, {\em Trends in PDE
  Constrained Optimization}, volume 165 of {\em International Series of
  Numerical Mathematics}, pages 193--211. Springer International Publishing,
  2014.

\bibitem{DeZo01}
M.~C. Delfour and J.~Zol{\'{e}}sio.
\newblock {\em Shapes and Geometries: Analysis, Differential Calculus and
  Optimization}.
\newblock Adv. Des. Control 4. SIAM, Philadelphia, 2001.

\bibitem{Do96}
W.~D{\"{o}}rfler.
\newblock A convergent adaptive algorithm for {P}oisson's equation.
\newblock {\em SIAM J. Numer. Anal.}, 33(3):1106--1124, June 1996.

\bibitem{EEn03}
W.~E and B.~Engquist.
\newblock The heterogeneous multiscale methods.
\newblock {\em Commun. Math. Sci.}, 1(1):87--132, 2003.

\bibitem{EEn05}
W.~E and B.~Engquist.
\newblock The heterogeneous multi-scale method for homogenization problems.
\newblock In {\em Multiscale Methods in Science and Engineering}, volume~44 of
  {\em Lecture Notes in Computational Science and Engineering}, pages 89--110.
  Springer Berlin Heidelberg, 2005.

\bibitem{EEnHu03}
W.~E, B.~Engquist, and Z.~Huang.
\newblock Heterogeneous multiscale method: A general methodology for multiscale
  modeling.
\newblock {\em Physical Review B}, 67(9):1--4, March 2003.

\bibitem{EMiZh05}
W.~E, P.~Ming, and P.~Zhang.
\newblock Analysis of the heterogeneous multiscale method for elliptic
  homogenization problems.
\newblock {\em J. Amer. Math. Soc.}, 18(1):121--156, 2005.

\bibitem{FrMu86}
G.~A. Francfort and F.~Murat.
\newblock Homogenization and optimal bounds in linear elasticity.
\newblock {\em Arch. Rational Mech. Anal.}, 94(4):307--334, 1986.

\bibitem{GeRu15}
B.~Geihe and M.~Rumpf.
\newblock A posteriori error estimates for sequential laminates in shape
  optimization.
\newblock {\em Discrete and Continuous Dynamical Systems - Series S},
  9(5):1377--1392, 2016.

\bibitem{GiCh87}
L.~Gibiansky and A.~Cherkaev.
\newblock Microstructures of composites of extremal rigidity and exact
  estimates of the associated energy density.
\newblock {\em Ioffe Physicotechnical Institute}, 1115, 1987.

\bibitem{GrKo95}
Y.~Grabovsky and R.~V. Kohn.
\newblock Microstructures minimizing the energy of a two phase elastic
  composite in two space dimensions. {I}. {T}he confocal ellipse construction.
\newblock {\em J. Mech. Phys. Solids}, 43(6):933--947, 1995.

\bibitem{Ha62}
Z.~Hashin.
\newblock The elastic moduli of heterogeneous materials.
\newblock {\em Trans. ASME Ser. E. J. Appl. Mech.}, 29:143--150, 1962.

\bibitem{HaKoLe10}
J.~Haslinger, M.~Ko\v{c}vara, G.~Leugering, and M.~Stingl.
\newblock Multidisciplinary free material optimization.
\newblock {\em SIAM Journal on Applied Mathematics}, 70(7):2709--2728, 2010.

\bibitem{HeOh09}
P.~Henning and M.~Ohlberger.
\newblock The heterogeneous multiscale finite element method for elliptic
  homogenization problems in perforated domains.
\newblock {\em Numerische Mathematik}, 113, Issue 4:601--629, october 2009.

\bibitem{Jikov1994}
V.~V. Jikov, S.~M. Kozlov, and O.~A. Ole{\u\i}nik.
\newblock {\em Homogenization of differential operators and integral
  functionals}.
\newblock Springer-Verlag, Berlin, 1994.

\bibitem{JogHaberBendsoe1994}
C.~S. Jog, R.~B. Haber, and M.~P. Bends{\o}e.
\newblock Topology design with optimized, self-adaptive materials.
\newblock {\em Internat. J. Numer. Methods Engrg.}, 37(8):1323--1350, 1994.

\bibitem{JoHa96}
C.~S. Jog and R.~B. Haber.
\newblock Stability of finite element models for distributed-parameter
  optimization and topology design.
\newblock {\em Computer Methods in Applied Mechanics and Engineering},
  130(3--4):203--226, 1996.

\bibitem{BoJo98}
F.~Jouve and E.~Bonnetier.
\newblock Checkerboard instabilities in topological shape optimization
  algorithms.
\newblock In {\em Proceedings of the Conference on Inverse Problems, Control
  and Shape Optimization (PICOF'98), Carthage (1998)}, 1998.

\bibitem{KaDeGa14}
L.~Kaland, J.~C. De~Los~Reyes, and N.~R. Gauger.
\newblock One-shot methods in function space for pde-constrained optimal
  control problems.
\newblock {\em Optimization Methods and Software}, 29(2):376--405, 2014.

\bibitem{Ka13}
L.~Kaland.
\newblock {\em The one-shot method : function space analysis and algorithmic
  extension by adaptivity}.
\newblock Dissertation, RWTH Aachen, 2013.

\bibitem{KiVe13}
B.~Kiniger and B.~Vexler.
\newblock A priori error estimates for finite element discretizations of a
  shape optimization problem.
\newblock {\em ESAIM: Mathematical Modelling and Numerical Analysis},
  47:1733--1763, 11 2013.

\bibitem{KoLe13}
P.~Kogut and G.~Leugering.
\newblock Matrix-{V}alued {$L^1$}-{O}ptimal {C}ontrols in the {C}oefficients of
  {L}inear {E}lliptic {P}roblems.
\newblock {\em Z. Anal. Anwend.}, 32(4):433--456, 2013.

\bibitem{KoLi88}
R.~V. Kohn and R.~Lipton.
\newblock Optimal bounds for the effective energy of a mixture of isotropic,
  incompressible, elastic materials.
\newblock {\em Arch. Rational Mech. Anal.}, 102(4):331--350, 1988.

\bibitem{LeMeVe13}
D.~Leykekhman, D.~Meidner, and B.~Vexler.
\newblock Optimal error estimates for finite element discretization of elliptic
  optimal control problems with finitely many pointwise state constraints.
\newblock {\em Computational Optimization and Applications}, 55(3):769--802,
  2013.

\bibitem{LiCa15}
J.~Liu, L.~Cao, N.~Yan, and J.~Cui.
\newblock Multiscale approach for optimal design in conductivity of composite
  materials.
\newblock {\em SIAM J. Numer. Anal.}, 53(3):1325--1349, 2015.

\bibitem{LuCh86}
K.~A. Lurie and A.~V. Cherkaev.
\newblock Effective characteristics of composite materials and the optimal
  design of structural elements.
\newblock {\em Adv. in Mech.}, 9(2):3--81, 1986.

\bibitem{Mi02}
G.~W. Milton.
\newblock {\em The Theory of Composites}.
\newblock Cambridge University Press, 2002.

\bibitem{MoNoPa10}
P.~Morin, R.~Nochetto, M.~Pauletti, and M.~Verani.
\newblock Adaptive sqp method for shape optimization.
\newblock In G.~Kreiss, P.~L{\"{o}}tstedt, A.~M{\r{a}}lqvist, and M.~Neytcheva,
  editors, {\em Numerical Mathematics and Advanced Applications 2009}, pages
  663--673. Springer Berlin Heidelberg, 2010.

\bibitem{MoNoPa12}
P.~Morin, R.~H. Nochetto, M.~S. Pauletti, and M.~Verani.
\newblock Adaptive finite element method for shape optimization.
\newblock {\em ESAIM: Control, Optimisation and Calculus of Variations},
  18:1122--1149, 10 2012.

\bibitem{MuTa85}
F.~Murat and L.~Tartar.
\newblock Calcul des variations et homog\'en\'eisation.
\newblock In {\em Homogenization methods: theory and applications in physics
  ({B}r\'eau-sans-{N}appe, 1983)}, volume~57 of {\em Collect. Dir. \'Etudes
  Rech. \'Elec. France}, pages 319--369. Eyrolles, Paris, 1985.

\bibitem{OdVe00}
J.~T. Oden and K.~Vemaganti.
\newblock Adaptive modeling of composite structures: Modeling error estimation.
\newblock {\em International Journal for Computational Civil and Structural
  Engineering}, 1:1--16, 2000.

\bibitem{OdVe00a}
J.~T. Oden and K.~S. Vemaganti.
\newblock Estimation of local modeling error and goal-oriented adaptive
  modeling of heterogeneous materials. {I}. {E}rror estimates and adaptive
  algorithms.
\newblock {\em J. Comput. Phys.}, 164(1):22--47, 2000.

\bibitem{Oh05}
M.~Ohlberger.
\newblock A posterior error estimates for the heterogenoeous mulitscale finite
  element method for elliptic homogenization problems.
\newblock {\em SIAM Multiscale Mod. Simul.}, 4(1):88--114, 2005.

\bibitem{PaTr08}
O.~Pantz and K.~Trabelsi.
\newblock A post-treatment of the homogenization method for shape optimization.
\newblock {\em SIAM J. Control Optim.}, 47(3):1380--1398, 2008.

\bibitem{PrOd99}
S.~Prudhomme and J.~T. Oden.
\newblock On goal-oriented error estimation for elliptic problems: application
  to the control of pointwise errors.
\newblock {\em Comput. Methods Appl. Mech. Engrg.}, 176(1-4):313--331, 1999.
\newblock New advances in computational methods (Cachan, 1997).

\bibitem{RoFe94}
H.~Rodrigues and P.~Fernandes.
\newblock Topology optimal design of thermoelastic structures using a
  homogenization method.
\newblock {\em Control Cybernet.}, 23(3):553--563, 1994.
\newblock Shape design and optimization.

\bibitem{St03}
O.~Steinbach.
\newblock {\em Numerische {N}{\"{a}}herungsverfahren f{\"{u}}r elliptische
  {R}andwertprobleme: {F}inite {E}lemente und {R}andelemente}.
\newblock B.~G. Teubner, Wiesbaden, 2003.

\bibitem{SuKi91}
K.~Suzuki and N.~Kikuchi.
\newblock Layout optimization using the homogenization method.
\newblock In {\em Optimization of large structural systems, {V}ol.\ {I}, {II}
  ({B}erchtesgaden, 1991)}, volume 231 of {\em NATO Adv. Sci. Inst. Ser. E
  Appl. Sci.}, pages 157--175. Kluwer Acad. Publ., Dordrecht, 1993.

\bibitem{Ta85}
L.~Tartar.
\newblock Estimations fines des coefficients homog\'en\'eis\'es.
\newblock In {\em Ennio {D}e {G}iorgi colloquium ({P}aris, 1983)}, volume 125
  of {\em Res. Notes in Math.}, pages 168--187. Pitman, Boston, MA, 1985.

\bibitem{Ve04}
K.~Vemaganti.
\newblock Modelling error estimation and adaptive modelling of perforated
  materials.
\newblock {\em Internat. J. Numer. Methods Engrg.}, 59(12):1587--1604, 2004.

\bibitem{VeWo08}
B.~Vexler and W.~Wollner.
\newblock Adaptive finite elements for elliptic optimization problems with
  control constraints.
\newblock {\em SIAM Journal on Control and Optimization}, 47(1):509--534, 2008.

\bibitem{Vi94}
S.~Vigdergauz.
\newblock Two-dimensional grained composites of extreme rigidity.
\newblock {\em Journal of Applied Mechanics}, 61(2):390--394, 1994.

\bibitem{Wa02}
A.~W{\"a}chter.
\newblock {\em An Interior Point Algorithm for Large-Scale Nonlinear
  Optimization with Applications in Process Engineering}.
\newblock Phd thesis, Carnegie Mellon University, 2002.

\bibitem{WaBi06}
A.~W{\"a}chter and L.~T. Biegler.
\newblock {On the Implementation of a Primal-Dual Interior Point Filter Line
  Search Algorithm for Large-Scale Nonlinear Programming}.
\newblock {\em Mathematical Programming}, 106(1):25--57, 2006.

\bibitem{Wo10}
W.~Wollner.
\newblock Goal-oriented adaptivity for optimization of elliptic systems subject
  to pointwise inequality constraints: Application to free material
  optimization.
\newblock {\em PAMM}, 10(1):669--672, 2010.

\end{thebibliography}

\end{document}